\documentclass[a4paper,11pt]{amsart}
\usepackage[english]{babel}
\usepackage[latin1]{inputenc}
\usepackage{amsfonts}
\usepackage{amssymb}
\usepackage{amsmath}
\usepackage{amsthm}
\usepackage{bbm}
\usepackage{vmargin}
\usepackage{epsfig}
\usepackage{fancyhdr}
\newcommand{\R}{\mathbf{R}}
\newcommand{\N}{\mathbf{N}}
\newcommand{\Z}{\mathbf{Z}}
\newcommand{\Q}{\mathbf{Q}}

\newcommand{\frakR}{\mathfrak{R}}
\newcommand{\calR}{\mathcal{R}}
\newcommand{\frakT}{\mathfrak{T}}

\newcommand{\M}{{\mathcal{M}}}

\newcommand{\calS}{\mathcal{S}}
\newcommand{\calK}{\mathcal{K}}
\newcommand{\calL}{\mathcal{L}}

\renewcommand{\phi}{\varphi}
\renewcommand{\epsilon}{\varepsilon}
\renewcommand{\newline}{\\ \strut}
\theoremstyle{plain}
\newtheorem{theorem}{Theorem}[section]
\newtheorem{lemma}[theorem]{Lemma}
\newtheorem{cor}[theorem]{Corollary}

\newtheorem{prop}[theorem]{Proposition}

\newtheorem{conj}[theorem]{Conjecture}

\numberwithin{equation}{section}
\begin{document}
\thanks{The author was supported by a stipend (EliteForsk) from The Danish Agency for Science, Technology and Innovation.}
\title{Divisor problems and the pair correlation for the fractional parts of $n^2\alpha$}    
\author{Jimi L. Truelsen}
\address{Department of Pure Mathematics and Mathematical Statistics, University of Cambridge, Wilberforce Road, Cambridge CB3 0WB, United Kingdom}
\email{lee@imf.au.dk}
\subjclass[2000]{Primary 11B05, 11K31; Secondary 11J54}
\date{\today}
\begin{abstract}
Z. Rudnick and P. Sarnak have proved that the pair correlation for the fractional parts of $n^2 \alpha$ is Poissonian for almost all $\alpha$. However, they were not able to find a specific $\alpha$ for which it holds. We show that the problem is related to the problem of determining the number of $(a,b,r) \in \N^3$ such that $a \le M$, $b \le N$, $r \le K$ and $p ab \equiv r (q)$ for $p$ and $q$ coprime. With suitable assumptions on the relative size of $K$, $M$, $N$ and $q$ one should expect there to be $KMN/q$ such triples asymptotically and we will show that this holds on average.
\end{abstract}
\maketitle
\section{Introduction}
For $t \in \R$ and $q \in \N$ let
\begin{align*}
\Vert t \Vert_q = \inf_{n \in \Z} \vert t - qn \vert,
\end{align*}
and set $\Vert \cdot \Vert = \Vert \cdot \Vert_1$. Clearly $\Vert \cdot \Vert_q$ defines a norm on $\R/q\Z$. For a sequence $\{a_n \}_1^\infty \subset \R/\Z$, $x > 0$ and $N \in \N$ we define
\begin{align*}
R_2(x,N,\{a_n\}_{1}^\infty) = N^{-1}\#\left\{(m,n)\in \N^2 \big| m,n \le N, \ n \ne m, \ \Vert a_m - a_n \Vert \le \frac{x}{N}\right\}.
\end{align*}
We say that the pair correlation for $\{a_n\}_{1}^\infty$ is Poissonian if for every $x > 0$ we have that
\begin{align*}
\lim_{N \to \infty} R_2(x,N,\{a_n\}_{1}^\infty) = 2x.
\end{align*}
Note that the limit is not uniform in $x$. We will be particularly interested in the case where $a_n$ equals the fractional parts of $n^2 \alpha$ for $\alpha$ irrational. The spacings between the elements of this sequence correspond to the spacings between the energy levels of the boxed oscillator in quantum mechanics \cite{berry77}. We define (by an obvious abuse of notation)
\begin{align*}
R_2(x,N,\alpha) = R_2(x,N,\{n^2\alpha \}_{1}^\infty) = N^{-1}\#\left\{(m,n) \mid m,n \le N, \ n \ne m, \ \Vert m^2 \alpha - n^2 \alpha \Vert \le \frac{x}{N}\right\}.
\end{align*}
Clearly we may as well assume that $0 < \alpha < 1$. We will be interested in $\alpha$ with certain Diophantine properties. We say that an irrational number $\alpha$ is of type $\kappa$ if
\begin{align*}
\left\vert \alpha - \frac{p}{q} \right\vert \gg \frac{1}{q^\kappa}.
\end{align*}
for all $p \in \Z$ and $q \in \N$. We say that $\alpha$ is ``Diophantine'' if $\alpha$ is of type $2+\epsilon$ for any $\epsilon > 0$. In particular all real, irrational algebraic numbers are Diophantine (Roth's theorem -- see Theorem 5.7.1 in \cite{miller06}). Note also that almost all $\alpha$ (with respect to the Lebesgue measure) are Diophantine. To see this we use the identity of sets
\begin{align*}
\{ \beta \in \R \mid \beta \ {\rm is \ not \ Diophantine} \} = \bigcup_{n=1}^\infty \bigcup_{l \in \Z} \bigcap_{k=1}^\infty \bigcup_{q=k}^\infty \bigcup_{p=1}^q \left[l+\frac{p}{q}-\frac{1}{q^{2+1/n}},l+\frac{p}{q}+\frac{1}{q^{2+1/n}} \right].
\end{align*}
Let $\calL$ denote the Lebesgue measure on the real line. We see that
\begin{align*}
\calL\left( \bigcup_{p=1}^q \left[l+\frac{p}{q}-\frac{1}{q^{2+1/n}},l+\frac{p}{q}+\frac{1}{q^{2+1/n}} \right] \right) = 2q^{-(1+1/n)}.
\end{align*}
Since
\begin{align*}
\sum_{q=1}^\infty q^{-(1+1/n)} < \infty
\end{align*}
it follows from the Borel-Cantelli lemma that
\begin{align*}
\calL\left( \bigcap_{k=1}^\infty \bigcup_{q=k}^\infty \bigcup_{p=1}^q \left[l+\frac{p}{q}-\frac{1}{q^{2+1/n}},l+\frac{p}{q}+\frac{1}{q^{2+1/n}} \right] \right) = 0.
\end{align*}
Thus the set of non-Diophantine real numbers is a null set.

It is a classical result due to H. Weyl \cite{weyl16} that the sequence $n^d \alpha$ is equidistributed modulo $1$ for any integer $d \ge 1$. However, it is not true that the pair correlation for the fractional parts of $n^d \alpha$, $d \ge 2$ is Poissonian for all irrational $\alpha$ (for $d=1$ it is never the case -- see Exercise 12.6.3 in \cite{miller06}). A simple construction shows (see \cite{rudnick98} p. 62) that $\alpha$ must be at least of type $d + 1$.

Z. Rudnick and P. Sarnak have proved \cite[Theorem~1]{rudnick98} that the pair correlation for the fractional parts of $n^d \alpha$ is Poissonian for almost all $\alpha$. Subsequently J. Marklof and A. Str\"{o}mbergsson \cite{marklof03}, and D. R. Heath-Brown \cite{heathbrownPP} have given different proofs in the case $d = 2$. However, one does not know of any specific $\alpha$ for which it holds, but Rudnick and Sarnak \cite{rudnick98} and Heath-Brown \cite{heathbrownPP} made the following conjecture:
\begin{conj}\label{rs_conj}
Assume $\alpha$ is Diophantine. Then the pair correlation for the fractional parts of $n^2 \alpha$ is Poissonian.
\end{conj}
Furthermore, in \cite{heathbrownPP} Heath-Brown was able to show (using a lattice point strategy) that for $\alpha$ of type $9/4$
\begin{align}\label{HB_result}
R_2(x,N,\alpha) = 2x + O(x^{7/8}),
\end{align}
whenever $1 \le x \le \log N$, where the constant implied depends on $\alpha$. This supports Conjecture \ref{rs_conj} and suggests that perhaps the condition on the Diophantine approximation in the conjecture can be relaxed to some extend.

We remark that the $m$-level correlation for the fractional parts of $n^2 \alpha$ has been studied by Rudnick, Sarnak and Zaharescu in \cite{rudnick01} and by Zaharescu in \cite{zaharescu03}. It is not known if the fractional parts of $n^2 \alpha$ for almost all $\alpha$ have Poissonian behavior, i.e. have the same distribution as a sequence of independent and uniformly distributed random variables, but it is expected (cf. the conjecture on page 38 in \cite{rudnick01}).

In this paper we will only be concerned with Conjecture \ref{rs_conj} (not higher level correlations). We suggest a line of attack that is based on the study of the function
\begin{align*}
\tau_{M,N}(m) = \# \{(a,b) \in \N^2 \mid  a \le M, \ b \le N, \ ab = m\},
\end{align*}
where $m \in \N$ and $M,N \ge 1$. We also define $\tau_M^* = \tau_{M,M}$. We make the following conjecture:
\begin{conj}\label{main_conj}
Let $K,M,N \ge 1$ with $M \asymp N$ (i.e. $C_1 N\le M \le C_2 N$) and $K \ge N^\eta$ for some $\eta > 0$. Assume also that $q \le N^{2-\delta}$ for some $\delta > 0$ and $(q,\rho) = 1$. Then
\begin{align*}
\sum_{r \le K} \sum_{m\equiv \rho r(q)} \tau_{M,N}(m) \sim \frac{KMN}{q}
\end{align*}
as $N \to \infty$ uniformly in $M$, $K$, $q$ and $\rho$. The rate of convergence may depend on $\eta$, $\delta$, $C_1$ and $C_2$.
\end{conj}
Conjecture \ref{main_conj} has applications to the pair correlation problem at hand. We will show that:
\begin{prop}
Conjecture \ref{main_conj} implies that the pair correlation for the fractional parts of $n^2\alpha$ is Poissonian for any $\alpha$ of type $3-\delta$ for any $\delta \in (0,1)$.
\end{prop}
This is an immediate consequence of Proposition \ref{tauConj_implies_pc}. As mentioned previously the pair correlation for the fractional parts of $n^2 \alpha$ is not Poissonian if $\alpha$ is not of type $3$. Conjecture \ref{main_conj} claims that $3-\delta$ is sufficient.

Conjecture \ref{main_conj} seems bold but natural. Indeed the conjecture provably holds if $q$ is smaller than $N^{1-\delta}$ (see Proposition \ref{triv_prop} below). However, it turns out that we need $q \ge N^{3/2+\delta}$ for our purpose. We can actually obtain partial results for larger $q$ as well based on a lattice point approach using the ideas of Heath-Brown \cite{heathbrownPP}. Before we can state the result we introduce some terminology. We say that a rational number $p/q$ is of type $(e,\calK)$ if
\begin{align*}
\left\vert \frac{p}{q}-\frac{u}{v}\right\vert \ge \frac{1}{\calK v^e}
\end{align*}
for any rational number $u/v$ with $u/v \ne p/q$. One easily checks that if $\alpha$ is an irrational number of type $e$ then there exists $\calK > 0$ such that the convergents will be of type $(e,\calK)$ from some step. 

Modifying the proof of (\ref{HB_result}) we prove the following ($\tau$ denotes the ordinary divisor function):
\begin{theorem}\label{lattice}
Let $K,M,N \ge 3$ with $M \asymp N$ and let $\gamma \in (0,1)$. Assume that
\begin{align}\label{lattice_tech_cond}
q^{1+\delta} \le \left( \frac{N^2}{K} \right)^{1/(1+\gamma)}
\end{align}
for some $\delta > 0$ and $KN/q \ge 1$. Then
\begin{align*}
\sum_{\vert r \vert \le K} \sum_{pm\equiv r (q)} \tau_{M,N}(m) = \frac{2KMN}{q} + O\biggl(N(KN/q)^{7/8}+\frac{N^2}{q}\biggl(\tau(q)^2(\log N)^3+\frac{K(\log \log N)^2}{(\log N)^{1/4}}\biggr)\biggr)
\end{align*}
uniformly in $M$, $K$, $p$ and $q$ for $p/q$ of type $(2+\gamma,\calK)$.
\end{theorem}

It is well known (see e.g. \cite{heathbrown79}) that one expects that
\begin{align}\label{tau_ap}
\sum_{\substack{n \le x\\ n\equiv r (q)}} \tau(n) \sim \frac{x}{q^2}\log x \sum_{d\mid(q,r)}\sum_{c\mid \frac{q}{d}}dc \mu\left(\frac{q}{dc}\right)
\end{align}
as $x \to \infty$ for $q \le x^{1-\delta}$ for some $\delta > 0$. Average results supporting this conjecture have been considered by Banks, Heath-Brown and Shparlinski \cite{banks05}, and Blomer \cite{blomer08}. If we adapt (\ref{tau_ap}) to $\tau_{M,N}$ we should expect that
\begin{align}\label{tauMN_ap}
\sum_{n\equiv r (q)} \tau_{M,N}(n) \sim \frac{MN}{q^2} \sum_{d\mid(q,r)}\sum_{c\mid \frac{q}{d}}dc \mu\left(\frac{q}{dc}\right)
\end{align}
for $M \asymp N$ and $q \le N^{2-\delta}$. Note that
\begin{align}\label{mu_phi_id}
\sum_{d\mid(q,r)}\sum_{c\mid \frac{q}{d}}dc \mu\left(\frac{q}{dc}\right) = \sum_{d\mid(q,r)} d \phi\left( \frac{q}{d}\right).
\end{align}
It has been proved by Linnik and Vinogradov \cite{linnik} that
\begin{align}\label{lv_thm}
\sum_{\substack{m \le x\\ m \equiv r(q)}} \tau(m) \ll \frac{\phi(q)x \log x}{q^2}
\end{align}
for $q \le x^{1-\delta}$ and $(r,q)=1$, where the constant implied depends on $\delta > 0$ only. In view of Conjecture \ref{main_conj} and (\ref{lv_thm}) it would be interesting to find upper bounds for
\begin{align*}
\sum_{m\equiv r (q)} \tau_N^*(m).
\end{align*}
Heath-Brown \cite{heathbrownPP} suggested the following conjecture which is the analogue of (\ref{lv_thm}) for $\tau_N^*$:
\begin{conj}\label{HB_conj}
Let $\delta \in (0,1)$. Then
\begin{align*}
\sum_{m\equiv  r (q)} \tau_N^*(m) \ll \frac{\phi(q)N^2}{q^2}
\end{align*}
uniformly for $(r,q) = 1$ and $q \le N^{2-\delta}$, where the constant implied depends only on $\delta$.
\end{conj}
Using the work of M. Nair and G. Tenenbaum \cite{nair98} we prove an upper bound for the sum in Conjecture \ref{HB_conj}.
\begin{prop}\label{tenenbaum_result}
Let $q \le N^{2-\delta}$. Then
\begin{align*}
\sum_{m\equiv r (q)} \tau_N^*(m) \ll \frac{N^2}{\phi(q)}e^{\sqrt{(2+\varepsilon)(\log\log N)(\log\log\log N)}} 
\end{align*}
uniformly for $(r,q)=1$ for any $\epsilon > 0$. 
\end{prop}
Note that the estimate in the proposition above is off by less than a factor of $(\log N)^\epsilon$ compared to Conjecture \ref{HB_conj} since $q/\phi(q) \ll \log\log q$.

The function $\tau_{M,N}$ is complicated. There is another similar function of interest
\begin{align*}
\tau_M(m) = \# \{ d \in \N \mid  d \le M, \ d \mid m\}.
\end{align*}
The function $\tau_M$ is in many ways simpler than $\tau_{M,N}$. The estimate corresponding to Conjecture \ref{HB_conj} holds. More precisely we prove:
\begin{theorem}\label{tau_N_thm}
Let $0 < \delta \le 1$, $0 < \epsilon < \frac{1}{8}$, $0 < \kappa$ and $2 \le N$. Assume also that $N \ge q^{\kappa}$. Then
\begin{align*}
\sum_{\substack{x < n \le x + y\\ n\equiv r (q)}} \tau_N(n) \ll \frac{y\phi(q)\log N}{q^2}
\end{align*}
uniformly for $N$, $(r,q) = 1$, $x^{\frac{1+4 \epsilon\delta}{1+\delta}} \le y \le x$, $x \ge c_0 q^{1+\delta}$, where $c_0$ and the constant implied depends at most on $\delta$, $\epsilon$ and $\kappa$. 
In particular
\begin{align*}
\sum_{\substack{m\equiv r (q)\\ m \le x}} \tau_N (m) \ll \frac{\phi(q)x \log N}{q^2} .
\end{align*}
\end{theorem}
Note that with $N = x+y$ we obtain
\begin{align*}
\sum_{\substack{m\equiv r (q)\\ x < m \le x+y}} \tau (m) \ll \frac{y\phi(q) \log x}{q^2}.
\end{align*}
This extension of (\ref{lv_thm}) was also obtained by P. Shiu \cite{shiu}.

Finally we show that Conjecture \ref{main_conj} and Conjecture \ref{HB_conj} hold on average. Indeed we start by proving that (\ref{tauMN_ap}) holds for most values of $q$ and $r$ if $(q,r)$ is small:
\begin{theorem}\label{tau_individ}
Let $\delta > 0$ and assume $M \asymp N$. Then
\begin{align*}
\sum_{(r,q) = k}\left(\sum_{m\equiv  r (q)} \tau_{M,N}(m) - \frac{MN}{q^2}\sum_{d \mid k}\sum_{c \mid \frac{q}{d}} dc \mu\left(\frac{q}{dc} \right) \right)^2 \ll \frac{N^{\max(7/2,4-\delta)+\epsilon}}{q}
\end{align*}
uniformly for $q \le N^{2-\delta}$ and $k \mid q$.
\end{theorem}
From Theorem \ref{tau_individ} we can deduce the following:
\begin{theorem}\label{main_conj_sup}
Let $M, N \ge 1$ with $M \asymp N$, $q \in \N$ and $K \ge N^{\eta}$ for some $\eta > 0$. Assume also that $q \le N^{2 -\delta}$ for some $\delta > 0$. Then
\begin{align*}
\sum_{(\rho, q) = 1} \left(\sum_{r \le K} \sum_{m\equiv \rho r(q)} \tau_{M,N}(m) - \frac{KMN}{q}\right)^2\ll \frac{K^2 N^4}{q}\left(q^{\varepsilon}\left(\frac{1}{q}+\frac{1}{K}\right)^2+N^{\max(-1/2,-\delta)+\epsilon}\right)
\end{align*}
for any $\epsilon > 0$.
\end{theorem}

In Proposition \ref{triv_prop} we show that Conjecture \ref{main_conj} holds for $q \le N^{1-\delta}$. Thus we can safely restrict our attention to the case where $q \ge \sqrt{N}$. We have the following corollary, which states that Conjecture \ref{main_conj} is true on average:
\begin{cor}
Let $M, N \ge 1$ with $M \asymp N$, $q \in \N$ and $K \ge N^{\eta}$ for some $\eta > 0$. Assume also that $\sqrt{N} \le q \le N^{2 -\delta}$ for some $\delta > 0$. Then
\begin{align*}
\frac{1}{\phi(q)}\sum_{(\rho, q) = 1} \left(\frac{q}{KMN}\sum_{r \le K} \sum_{m\equiv \rho r(q)} \tau_{M,N}(m) - 1\right)^2\ll N^{-\min(1/2,\delta,2\eta)+\epsilon}
\end{align*}
for any $\epsilon > 0$.
\end{cor}

The author would like to thank P. Sarnak for suggesting the problem of relating Conjecture \ref{rs_conj} to a divisor problem and D. R. Heath-Brown for generously sharing his ideas on the problem and providing crucial assistance at various stages. The author would also like to thank M. Risager and A. Str\"{o}mbergsson for comments on an earlier version of the manuscript.

\section{Reducing the Question to an Arithmetic Problem}
Set
\begin{align*}
S(x,N,\alpha) = \frac{\#\left\{(a,b) \in \N \times \Z \mid 1 \le a < 2N, \ 1 \le \vert b\vert \le N-\vert N-a \vert, \ 2 \mid a+b, \ \Vert ab\alpha \Vert \le \frac{x}{N}\right\}}{N}
\end{align*}
By factoring $m^2-n^2$ into $a = m+n$ and $b = m-n$ we see that
\begin{align*}
0 &\le S(x,N,\alpha) -R_2(x,N,\alpha)\\
& \le \frac{2}{N}+\frac{2}{N}\#\left\{n \in\N \mid n \le N, \ \Vert n^2 \alpha \Vert \le \frac{x}{N} \right\}\\
&\to 0
\end{align*}
as $N \to \infty$ (the difference between $S$ and $R_2$ is that in $S$ we do not exclude all the cases corresponding to $m$ or $n$ equal to $0$). This follows since the fractional parts of $n^2 \alpha$ becomes equidistributed in the unit interval. Thus if we want to study Poissonian behavior we may as well study $S(x,N,\alpha)$ rather than $R_2(x,N,\alpha)$.

From the elementary theory of continued fractions (see \cite{miller06} Chapter 7) we know that the convergents $p_n/q_n$ of $\alpha$ satisfy
\begin{align}
\left\vert \alpha -\frac{p_n}{q_n }\right\vert \le \frac{1}{q_n q_{n+1}}
\end{align}
and
\begin{align}\label{cond_ratio}
q_{n+1}p_n-p_{n+1}q_n = \pm 1.
\end{align}
Define $\frakR(y,N,p, q)$ by
\begin{align*}
\#&\left\{(a,b) \in \N \times \Z \mid 1 \le a < 2N, \ 1 \le \vert b\vert \le N-\vert N-a \vert, \ 2 \mid a+b, \ \Vert abp \Vert_{q} \le y \right\}.
\end{align*}
We have the following:
\begin{prop}\label{arithmetic_form}
Let $\alpha$ be irrational with convergents $p_n/q_n$. The pair correlation for the fractional parts of $n^2 \alpha$ is Poissonian if and only if there for all fixed $x > 0$ exists $\kappa > 0$ and a sequence $\{ n_k\}_1^\infty$ such that $N^{3+\kappa} \ll q_{n_N} q_{n_N+1}$ and
\begin{align*}
\frakR\left(\frac{xq_{n_N}}{N},N,p_{n_N},q_{n_N}\right) \sim 2xN
\end{align*}
as $N \to \infty$.
\end{prop}
\begin{proof}
Note that
\begin{align*}
\Vert ab\alpha \Vert \le \frac{x}{N}
\end{align*}
if and only if
\begin{align*}
\left\Vert abp_n + abq_n\left(\alpha-\frac{p_n}{q_n}\right)\right\Vert_{q_n} \le \frac{xq_n}{N}.
\end{align*}
Now
\begin{align*}
\left\vert  abq_n\left(\alpha-\frac{p_n}{q_n}\right)\right\vert \le \frac{N^2}{q_{n+1}},
\end{align*}
and this implies that
\begin{align*}
\frakR\left(\frac{xq_n}{N}-\frac{N^2}{q_{n+1}},N,p_n,q_n\right) \le NS(\alpha,N,x) \le \frakR\left(\frac{xq_n}{N}+\frac{N^2}{q_{n+1}},N,p_n,q_n\right).
\end{align*}
Assume $\frac{N^2}{q_{n+1}} =o\left(\frac{xq_n}{N}\right)$ as $N \to \infty$. Since $\frakR(y,N,p, q)$ is an increasing function of $y$ we conclude that for any $\varepsilon > 0$
\begin{align*}
\frakR\left(\frac{(x-\varepsilon)q_n}{N},N,p_n,q_n\right) \le NS(\alpha,N,x) \le \frakR\left(\frac{(x+\varepsilon)q_n}{N},N,p_n,q_n\right)
\end{align*}
for $N$ sufficiently large. From this the result follows easily.
\end{proof}
Now we have an arithmetic version of Conjecture \ref{rs_conj}. However, the constraints on $a$ and $b$ in the definition of $\frakR(y,N,p,q)$ are a bit complicated. We can split $\frakR(y,N,p,q)$ into some nicer pieces. The hope is that we can say something about these. This is where Conjecture \ref{main_conj} enters the picture as we will see below. First we make a (technical) conjecture:
\begin{conj}\label{div_conj}
Let $\calK, \lambda ,\zeta, c > 0$ be constants with $\lambda < 1$. Let $N \ge 1$ and assume $p/q \in \Q$ (with $(p,q)=1$) is of type $(2+\lambda,\calK)$ and $N^{\frac{3}{2+\lambda}}\le q \le N^{\frac{3(1+\lambda)}{2+\lambda}}$. Then
\begin{align*}
\sum_{\vert r \vert \le \frac{\zeta q}{N}} \sum_{p m\equiv r(q)} \tau_{cN,N}(m) \sim 2 \zeta c N
\end{align*}
as $N \to \infty$ uniformly in $q$ and $p$.
\end{conj}
Clearly Conjecture \ref{main_conj} implies Conjecture \ref{div_conj}. From the next proposition we may therefore conclude that Conjecture \ref{main_conj} implies Conjecture \ref{rs_conj}.
\begin{prop}\label{tauConj_implies_pc}
Assume Conjecture \ref{div_conj} holds with some $\lambda \in (0,1)$ and for any $\calK > 1$. For $\alpha$ of type $2+\beta$ with $\beta < \lambda$ the pair correlation for the fractional parts of $n^2\alpha$ is Poissonian. In particular Conjecture \ref{rs_conj} holds.
\end{prop}
\begin{proof}
Let us first ignore the technical conditions on $N$, $p$ and $q$.
We see that
\begin{gather}
\begin{split}\label{split_frakR}
\frakR&\left(\frac{xq}{N},N,p, q\right) = \#\left\{(a,b) \in \N^2 \mid a,b \le N, \ 2 \mid a+b, \ \Vert abp \Vert_{q} \le xq/N \right\}+\\
&\phantom{==}2\#\{(a,b)\in \N^2 \mid N < a < 2N, \ 0 < b \le 2N-a, \ 2 \mid a+b, \ \Vert abp \Vert_{q} \le xq/N \}+\\
&\phantom{==}\#\{a \in \N \mid a \le N, \Vert a^2 p \Vert_{q} \le xq/N \}.
\end{split}
\end{gather}
The last term is negligible as $N \to \infty$, since the sequence $\{n^2\alpha \}_1^\infty$ is equidistributed modulo $1$. We will see that the first two terms are of the same size (if we assume Conjecture \ref{div_conj}). We start by considering the first term. We define
\begin{align*}
\frakT_M(m) = \# \{(a,b)\in \N^2 \mid  a,b \le M, \ a \equiv b (2) , \ ab = m\}.
\end{align*}
Note that
\begin{align}\label{Ttau_rel}
\frakT_M(m) = \begin{cases} \tau_M^*(m) & {\rm if \ } m \equiv 1 (2)\\ 0 & {\rm if \ } m \equiv 2 (4)\\
\tau_{M/2}^*(m/4) & {\rm if \ } m \equiv 0 (4) \end{cases}.
\end{align}
Using this notation we see that (assuming that $N > 2x$)
\begin{align*}
\#\left\{(a,b) \in \N^2 \mid a,b \le N, \ 2 \mid a+b, \ \Vert abp \Vert_{q} \le xq/N \right\}= \sum_{\vert r \vert \le \frac{xq}{N}}\sum_{m \equiv r\overline{p} (q)}\frakT_N(m),
\end{align*}
where $\overline{p}$ is the inverse of $p$ modulo $q$. Using (\ref{Ttau_rel}) we can write this as
\begin{align*}
\sum_{\vert r \vert \le \frac{xq}{N}}\sum_{m \equiv r\overline{p} (q)}\frakT_N(m) = \sum_{\vert r \vert \le \frac{xq}{N}}\bigl(\sum_{2k+1 \equiv r\overline{p} (q)}\tau_N^*(2k+1)+\sum_{4l \equiv r\overline{p} (q)} \tau_{N/2}^*(l) \bigr).
\end{align*}
Furthermore we see that
\begin{align}\label{frakT_tau_star}
\sum_{\vert r \vert \le \frac{xq}{N}}\sum_{4l \equiv r\overline{p} (q)} \tau_{N/2}^*(l) = \begin{cases} \sum_{\vert r \vert \le \frac{xq}{N}}\sum_{l \equiv r\overline{4p} (q)} \tau_{N/2}^*(l) & {\rm if \ } q \equiv 1 (2)\\ \sum_{\vert r \vert \le \frac{xq}{2N}}\sum_{l \equiv r\overline{2p} (q/2)} \tau_{N/2}^*(l)  & {\rm if \ } 2 \parallel q\\
\sum_{\vert r \vert \le \frac{xq}{4N}}\sum_{l \equiv r\overline{p} (q/4)} \tau_{N/2}^*(l) & {\rm if \ } 4 \mid q \end{cases}.
\end{align}
By (\ref{frakT_tau_star}) Conjecture \ref{div_conj} implies that
\begin{align*}
\sum_{\vert r \vert \le \frac{xq}{N}}\sum_{4l \equiv r\overline{p} (q)} \tau_{N/2}^*(l) \sim \frac{1}{2}xN.
\end{align*}
Note also that
\begin{align*}
\sum_{2k+1 \equiv r\overline{p} (q)}\tau_N^*(2k+1) = \sum_{m \equiv r\overline{p} (q)}\tau_N^*(m)- \sum_{2m \equiv r\overline{p} (q)}\tau_N^*(2m).
\end{align*}
If $q$ is odd then
\begin{align*}
\sum_{2m \equiv r\overline{p} (q)}\tau_N^*(2m) = 2\sum_{m \equiv r\overline{2p} (q)}\tau_{N,N/2}(m)-\sum_{m \equiv r\overline{4p} (q)}\tau_{N/2}^*(m).
\end{align*}
Thus Conjecture \ref{div_conj} implies
\begin{align*}
\sum_{\vert r \vert \le \frac{xq}{N}}\sum_{2k+1 \equiv r\overline{p} (q)}\tau_N^*(2k+1) \sim \frac{1}{2}xN.
\end{align*}
Now assume $q$ is even. This implies that $p$ and $\overline{p}$ are odd. Thus $m$ is even if and only if $r$ is even. Hence
\begin{align*}
\sum_{\vert r \vert \le \frac{xq}{N}}\sum_{2m \equiv r\overline{p} (q)}\tau_N^*(2m) &= 2\sum_{\vert r \vert \le \frac{xq}{2N}}\sum_{m \equiv r\overline{p} (q/2)}\tau_{N,N/2}(m)-\\
&\phantom{=}\begin{cases} \sum_{\vert r \vert \le \frac{xq}{2N}}\sum_{m \equiv r\overline{2p} (q/2)} \tau_{N/2}^*(m) & {\rm if \ } 2 \parallel q\\ \sum_{\vert r \vert \le \frac{xq}{4N}}\sum_{m \equiv r\overline{p} (q/4)} \tau_{N/2}^*(m) & {\rm if \ } 4 \mid q \end{cases}.\\
&\sim \frac{3}{2}xN.
\end{align*}
Thus we conclude that 
\begin{align*}
\sum_{\vert r \vert \le \frac{xq}{N}}\sum_{m \equiv r\overline{p} (q)}\frakT_N(m) \sim xN.
\end{align*}
Now we consider the second term in (\ref{split_frakR}) and we set
\begin{align*}
\mathfrak{S} = \#\left\{(a,b)\in \N^2 \mid N < a < 2N, \ 0 < b \le 2N-a, \ 2 \mid a+b, \ \Vert abp \Vert_{q} \le xq/N \right\}.
\end{align*}
For $k,l_1,l_2 \in \N$ such that $l_1 + l_2 \le 2^k$ we define
\begin{align*}
T_{l_1,l_2}(k) = \#\left\{(a,b)\in \N^2 \mid a \in I(k,l_1), b \in J(k,l_2), \ 2 \mid a+b, \ \Vert abp \Vert_{q} \le xq/N \right\},
\end{align*}
where
\begin{align*}
I(k,l_1) = (N(1+(l_1-1)/2^k) , N(1+l_1/2^k)]
\end{align*}
and
\begin{align*}
J(k,l_2) = (N(l_2-1)/2^k , Nl_2/2^k].
\end{align*}
As before we deduce (using Conjecture \ref{div_conj}) that
\begin{align}\label{T_rel}
T_{l_1,l_2}(k) \sim \frac{1}{4^k}xN
\end{align}
uniformly in $l_1$ and $l_2$ (since $k$ is fixed). Clearly we have
\begin{align*}
\sum_{\substack{l_1,l_2\\ l_1 + l_2 \le 2^k-1}}&T_{l_1,l_2}(k)\\
 &\le \#\left\{(a,b)\in \N^2 \mid N < a < 2N, \ 0 < b \le 2N-a, \ 2 \mid a+b, \ \Vert abp \Vert_{q} \le xq/N \right\}\\
&\le \sum_{\substack{l_1,l_2\\ l_1 + l_2 \le 2^k}}T_{l_1,l_2}(k).
\end{align*}
Recall that $\#\{(l_1,l_2)\in \N^2 \mid l_1 + l_2 \le m\} = \frac{m(m+1)}{2}$. Using (\ref{T_rel}) we see that
\begin{align*}
\frac{1-2^{-k}}{2}x \le \liminf_N \frac{\mathfrak{S}}{N} \le \limsup_N \frac{\mathfrak{S}}{N} \le \frac{1+2^{-k}}{2}x.
\end{align*}
Since this holds for any $k$ we must have $\mathfrak{S} \sim xN/2$ as desired. By Proposition \ref{arithmetic_form} it remains to prove that there exists $\kappa > 0$ such that for each $N$ sufficiently large we can choose $q_n$ and $q_{n+1}$ such that $N^{3+\kappa} \ll q_n q_{n+1}$. By Conjecture \ref{div_conj} we must take $q_n \le N^{\frac{3(1+\lambda)}{(2+\lambda)}}$. Recall that $\alpha$ is of type $2 + \beta$. Choose $n$ such that
\begin{align*}
q_n \le N^{\frac{3(1+\lambda)}{(2+\lambda)}} \le q_{n+1}.
\end{align*}
The condition that $\alpha$ is of type $2+\beta$ implies that
\begin{align*}
q_n q_{n+1} \ll q_n^{2+\beta}
\end{align*}
and hence
\begin{align*}
q_n q_{n+1} \gg q_{n+1}^{1+\frac{1}{1+\beta}} \ge N^{\frac{3(1+\lambda)}{(2+\lambda)}(1+\frac{1}{1+\beta})}.
\end{align*}
Thus we can choose $\kappa = \frac{\lambda-\beta}{(1+\beta)(2+\lambda)}$.
\end{proof}
It should be mentioned that there may be some loss in using the rational approximation at an early stage in Proposition \ref{arithmetic_form}. For this approach to work we must be able to work with $q \ge N^{3/2+\delta}$. In Theorem \ref{lattice} we can say something about values of $q$ that are slightly smaller. In the proof of (\ref{HB_result}) Heath-Brown was able to work with $\alpha$ rather than its convergents and only use the Diophantine approximation at the very end of the proof allowing the use of smaller values of $q$.

By condition (\ref{cond_ratio}) we see that the inverse of $p_n$ modulo $q_n$ is $\pm q_{n+1}$.  To begin with one could study
\begin{align*}
\sum_{\vert r \vert \le \frac{xq_n}{N}}\sum_{m \equiv rq_{n+1}(q_n)}\tau_N^*(m).
\end{align*}
Perhaps one can use this information to say more about the pair correlation problem for specific $\alpha$'s such as $\sqrt{2}$ or the golden ratio where the $q_n$'s are known.

\section{Preliminary Evidence for Conjecture \ref{main_conj}}
We will now explain why we should expect the asymptotics in Conjecture \ref{main_conj}. We try the ``naive'' approach. Define
\begin{align*}
\delta_d(n) = \begin{cases} 1 & {\rm if \ } d \mid n\\ 0 & {\rm if \ } d \nmid n \end{cases}.
\end{align*}
Assume $M,N \ge 2$. Clearly $\tau_{M,N}(m) = \sum_{d=\lceil m/N\rceil}^M \delta_d(m)$. Thus
\begin{align*}
\sum_{r \le K} \sum_{m\equiv \rho r (q)} \tau_{M,N}(m) &= \sum_{r \le K} \sum_{l \le (MN-[r\rho]_q)/q}\sum_{d=\lceil (ql+[r\rho]_q)/N\rceil}^M \delta_d(ql+[r\rho]_q)\\
&=\sum_{r \le K} \sum_{l \le (MN-[r\rho]_q)/q}\sum_{(ql+[r\rho]_q)/N \le d \le M} \delta_d(ql+[r\rho]_q)+O(N^2q^{-1+\epsilon})\\
&=\sum_{r \le K} \sum_{[r\rho]_q/N \le d \le M}\sum_{
0 \le l \le (Nd-[r\rho]_q)/q} \delta_d(ql+[r\rho]_q)+O(N^2q^{-1+\epsilon})
\end{align*}
where $[\cdot]_q$ denotes the remainder when dividing by $q$. Now the length of the $l$-interval can be much smaller than $d$ and this is where the approach fails. We should expect that
\begin{align*}
\sum_{[r\rho]_q/N \le d \le M}\sum_{
0 \le l \le (Nd-[r\rho]_q)/q} \delta_d(ql+[r\rho]_q)
\end{align*}
is heuristically
\begin{align*}
\sum_{\substack{[r\rho]_q/N \le d \le M\\ (d,q)\mid r}}\frac{Nd-[r\rho]_q}{dq}(d,q) =\frac{N}{q} \sum_{\substack{d \le M\\ (d,q)\mid r}}(d,q)+O\left(\sum_{\substack{d \le M\\ (d,q)\mid r}}\frac{(d,q)}{d}\right)+O\left(\frac{N}{q}\sum_{\substack{d \le q/N\\ (d,q)\mid r}}(d,q)\right),
\end{align*}
and it is the case if $q \le N^{1-\delta}$. Using Lemma \ref{assymp_lemma} below we see that the ``expected'' value of $\sum_{r \le K} \sum_{m\equiv \rho r (q)} \tau_{M,N}(m)$ is
\begin{align*}
\frac{KMN}{q}+O((K+N)Nq^{-1+\epsilon} + Kq^\epsilon).
\end{align*}

On numerous occasions we will use the fact that
\begin{align}\label{r_coprime}
\phi(x,q) = \# \{ n \in\N \mid n \le x, \ (n,q) =1 \} = \sum_{d \mid q} \mu(d)\left[ \frac{x}{d} \right] = \frac{\phi(q)x}{q} + O(\tau(q)). 
\end{align}
This implies that for $k \mid q$
\begin{align}\label{r_cop}
\# \{r\in \N \mid r \le K, \ (r,q) = k\} = \frac{K\phi(q/k)}{q}+O(q^\varepsilon).
\end{align}

\begin{lemma}\label{assymp_lemma}
Let $N,K \ge 1$ and $q \in \N$. Then
\begin{align}
\sum_{r \le K} \sum_{\substack{d \le N\\ (d,q) \mid r}} (d,q) = KN + O((K+N)q^{\varepsilon}),
\end{align}
\begin{align}
\sum_{r \le K} \sum_{\substack{d \le N\\ (d,q) \mid r}} \frac{(d,q)}{d} = K\log N +O(Kq^\varepsilon)+O(q^\varepsilon \log N)
\end{align}
and
\begin{align}
\sum_{r \le K} r \sum_{\substack{d \le N\\ (d,q) \mid r}} (d,q) = \frac{1}{2}NK^2 + O(K(K+N)q^\varepsilon)
\end{align}
for any $\varepsilon > 0$.
\end{lemma}
\begin{proof}
We see that
\begin{align*}
\sum_{r \le K} \sum_{\substack{d \le N\\ (d,q) \mid r}} (d,q) &= \sum_{\substack{k \mid q\\ k \le K}} k \# \{(r,d) \in \N^2 \mid r \le K, \ d \le N, \ (d,q) = k, \ k \mid r \}\\
&=\sum_{\substack{k \mid q\\ k \le K}} k \left[ \frac{K}{k}\right]\#\{d \in \N \mid d \le N, \ (d,q)=k\}\\
&=\sum_{\substack{k \mid q\\ k \le K}} k \left[ \frac{K}{k}\right] \sum_{a \mid q/k}\mu(a)\left[\frac{N}{ak} \right]\\
&=\sum_{\substack{k \mid q\\ k \le K}} k\left(\frac{K}{k}+O(1)\right)\sum_{a \mid q/k}\mu(a)\left(\frac{N}{ak}+O(1)\right)\\
&=KN\sum_{\substack{k \mid q\\ k \le K}}\sum_{a \mid q/k} \frac{\mu(a)}{ak}+O(Kq^\varepsilon)+O(Nq^\varepsilon)
\end{align*}
for any $\epsilon > 0$. Now
\begin{align*}
\sum_{\substack{k \mid q\\ k \le K}}\sum_{a \mid q/k} \frac{\mu(a)}{ak} = q^{-1}\sum_{\substack{k \mid q\\ k \le K}}\phi(q/k) =1-q^{-1}\sum_{\substack{k \mid q\\ k > K}}\phi(q/k)
\end{align*}
and
\begin{align*}
\sum_{\substack{k \mid q\\ k > K}}\phi(q/k) \le q\sum_{\substack{k \mid q\\ k > K}}\frac{1}{k} \le \frac{q}{K}\sum_{\substack{k \mid q\\ k > K}} 1 = O\left(\frac{q^{1+\epsilon}}{K}\right).
\end{align*}
Thus
\begin{align*}
\sum_{r \le K} \sum_{\substack{d \le N\\ (d,q) \mid r}} (d,q) = KN\left(1+O\left(\frac{q^{\epsilon}}{K}\right)\right)+O(Kq^\varepsilon)+O(Nq^\varepsilon) = KN + O((K+N)q^{\varepsilon}).
\end{align*}
Using partial summation we see that
\begin{align*}
\sum_{r \le K} \sum_{\substack{d \le N\\ (d,q) \mid r}} \frac{(d,q)}{d} &= N^{-1} \sum_{r \le K} \sum_{\substack{d \le N\\ (d,q) \mid r}} (d,q)+\int_1^N \frac{1}{t^2}\sum_{r \le K} \sum_{\substack{d \le t\\ (d,q) \mid r}} (d,q)dt\\
&= K + O((K/N+1)q^\varepsilon)+K\log N +O(Kq^\varepsilon)+O(q^\varepsilon \log N)\\
&= K\log N +O((K+1)q^\varepsilon)+O(q^\varepsilon \log N).
\end{align*}
The last part of the proposition is also proved using partial summation. We omit the details.
\end{proof}

The difficulty in proving Conjecture \ref{main_conj} obviously lies in dealing with the fact that we only consider a very small number of all the residue classes ($K$ can be much smaller than $q$) and at the same time $M$ and $N$ can be much smaller than $q$ (as pointed out earlier). Indeed the above shows:
\begin{prop}\label{triv_prop}
Let $K,M,N \ge 1$ with $M \asymp N$ and $K \ge N^\eta$ for some fixed $\eta > 0$. Assume also that $q \le N^{1-\delta}$ for some $\delta > 0$ and $(q,\rho) = 1$. Then
\begin{align*}
\sum_{r \le K} \sum_{m\equiv \rho r(q)} \tau_{M,N}(m) \sim \frac{KMN}{q}
\end{align*}
as $N \to \infty$ uniformly in $M$, $K$, $q$ and $\rho$.
\end{prop}
Conjecture \ref{main_conj} says that the asymptotic formula above still holds if we extend the range of $q$ to $q \le N^{2-\delta}$. In the same way we see that Conjecture \ref{HB_conj} holds for small values of $q$. More precisely we have
\begin{align*}
\sum_{m \equiv r (q)} \tau_N^*(m) \ll \frac{N^2 \phi(q)}{q^2}
\end{align*}
for $q \le N^{1-\delta}$, $(r,q) = 1$.

The following lemma will be useful in the next section. It shows that (\ref{tauMN_ap}) would imply the asymptotics in Conjecture \ref{div_conj}.
\begin{lemma}\label{Kq_lemma}
Let $q \in \N$. Then 
\begin{align*}
\sum_{r=1}^q \sum_{d\mid(q,r)}\sum_{c\mid \frac{q}{d}}dc \mu\left(\frac{q}{dc}\right) = q^2.
\end{align*}
If $K> 0$ then for any $\epsilon > 0$
\begin{align*}
\sum_{r\le K} \sum_{d\mid(q,r)}\sum_{c\mid \frac{q}{d}}dc \mu\left(\frac{q}{dc}\right) = Kq + O((K+q)q^{\epsilon}).
\end{align*}
\end{lemma}
\begin{proof} We see that
\begin{align*}
\sum_{r=1}^q \sum_{d\mid(q,r)}\sum_{c\mid \frac{q}{d}}dc \mu\left(\frac{q}{dc}\right) = \sum_{k \mid q}\phi(q/k)\sum_{d\mid k}\sum_{c\mid \frac{q}{d}}dc \mu\left(\frac{q}{dc}\right).
\end{align*}
Since the left hand side is a multiplicative function of $q$ it suffices to prove that
\begin{align}\label{prime_id}
\sum_{n = 0}^l \phi(p^{l-n})\sum_{m = 0}^n\sum_{j= 0}^{l-m}p^{m+j}\mu(p^{l-m-j}) = p^{2l}.
\end{align}
for a prime $p$. One easily checks that
\begin{align*}
\sum_{n = 0}^l \phi(p^{l-n})\sum_{m = 0}^n\sum_{j= 0}^{l-m}p^{m+j}\mu(p^{l-m-j}) &= p^{l-1}+p^l(1-p^{-1})\sum_{n = 0}^l \phi(p^{l-n}) (n+1)\\
&=p^{l-1}+p^l(1-p^{-1})\left(l+1+p^l(1-p^{-1})\sum_{n = 0}^{l-1} \frac{n+1}{p^n}\right).
\end{align*}
The identity (\ref{prime_id}) now follows since
\begin{align*}
\sum_{n=0}^{l-1} \frac{n+1}{p^n} = \frac{1-p^{-l-1}}{(1-p^{-1})^2}-\frac{l+1}{p^l(1-p^{-1})}.
\end{align*}

From the first part we deduce using (\ref{r_cop}) that
\begin{align*}
\sum_{r\le K} \sum_{d\mid(q,r)}\sum_{c\mid \frac{q}{d}}dc \mu\left(\frac{q}{dc}\right) &= \frac{K}{q}\sum_{\substack{k \mid q\\ k \le K}} \phi\left(\frac{q}{k}\right) \sum_{d\mid k}\sum_{c\mid \frac{q}{d}}dc \mu\left(\frac{q}{dc}\right)+O(q^{1+\varepsilon})\\
&=Kq-\frac{K}{q}\sum_{\substack{k \mid q\\ k > K}}\phi\left(\frac{q}{k}\right) \sum_{d\mid k}d\phi\left(\frac{q}{d}\right)+O(q^{1+\varepsilon})\\
&=Kq+O(q^{1+\varepsilon}).
\end{align*}
\end{proof}

\section{Average Results}
From \cite{friedlander85} we know how to count elements of an arithmetic progression in a given interval:
\begin{lemma}\label{friedlander}
Let $a < b$, $r \in \Z$ and $H \in \N$ with $H < q$. Then
\begin{align*}
\sum_{\substack{a < m \le b\\ m \equiv r (q)}} 1 = \frac{b-a}{q} + \sum_{1 \le \vert h \vert \le H} C_{a,b}(h)e(-hr/q)+O\left(\theta_{H}((a-r)/q) + \theta_{H}((b-r)/q) \right),
\end{align*}
where
\begin{align*}
\theta_H(s) = \begin{cases} \min(1,1/(H \Vert s\Vert)) & {\textit{if}} \ s \notin \Z\\ 1 & {\textit{if}} \ s \in \Z \end{cases}
\end{align*}
and
\begin{align*}
C_{a,b}(n) = \frac{e(bn/q) -e(an/q)}{2\pi i n}.
\end{align*}
\end{lemma}
Note that $H_1 \le H_2$ implies that $\theta_{H_1}(s) \ge \theta_{H_2}(s)$. It will be convenient to set
\begin{align*}
E_H(a,b,r,q) = \theta_{H}((a-r)/q) + \theta_{H}((b-r)/q).
\end{align*}
In \cite{banks05} the following lemma was proved:
\begin{lemma}\label{banks}
Let $(r,q) = 1$, $H$ as in Lemma \ref{friedlander} and $t \in \R$. Then
\begin{align*}
\sum_{c \in \Z_q} \theta_{H}((t-rc)/q) \ll \frac{q^{1+\epsilon}}{H}
\end{align*}
for any $\epsilon > 0$, where the constant implied depends at most on $\epsilon$.
\end{lemma}
We will use these results to prove a technical lemma. Define
\begin{align*}
I(m,M,\alpha,\beta,b,t) = \left\{ x \in \N \mid \max(m,(\alpha+t)/b) \le x \le  \min(M,(\beta+t)/b) \right\}
\end{align*}
and
\begin{align*}
J(m,M,\alpha,\beta,b,t) = \left\{ x \in \N \mid \max(m,(\alpha+t)/b) < x \le  \min(M,(\beta+t)/b) \right\}.
\end{align*}
\begin{lemma}\label{error_lemma_i}
Let $N \ge 1$. Then
\begin{align*}
\sum_{\substack{a \le N_1\\ b \le N_2}} &\sum_{c \in \Z} \# \left\{ x\in I(1,N_3,a,N_4a,b,r+cq) \mid xb \equiv r + cq (a)\right\}\\
&=\sum_{\substack{a \le N_1\\ b \le N_2}}\sum_{\substack{c \in \Z\\ (a,b) \mid cq+r }} \frac{(a,b)}{a}\int_0^{N_3}1_{\{t \in \R \mid (r+cq)/b \le t \le (N_4 a + r + cq)/b\}}(t)dt + O\left(\frac{N^{7/2+\epsilon}}{q}\right)
\end{align*}
for $N_i \asymp N$, $q\le N^{2-\delta}$ and $\vert r \vert \ll N^4$. In the first sum $a,b \in \N$.
\end{lemma}
\begin{proof}
We see that
\begin{align*}
\sum_{\substack{a \le N_1\\ b \le N_2}} &\sum_{c \in \Z} \# \left\{ x\in I(1,N_3,a,N_4a,b,r+cq) \mid xb \equiv r + cq (a)\right\}\\
&=\sum_{\substack{a \le N_1\\ b \le N_2}} \sum_{c \in \Z} \# \left\{ x\in J(0,N_3,0,N_4a,b,r+cq) \mid xb \equiv r + cq (a)\right\}.
\end{align*}
Note that we can assume that $c \in \widetilde{I} = [-(M^2+r)/q,(M^2-r)/q]$ since the sum is zero otherwise. Note that the length of $\widetilde{I}$ is $2M^2/q$. We now apply Lemma \ref{friedlander} and get
\begin{align*}
\sum_{\substack{a \le N_1\\ b \le N_2}}&\sum_{c \in \Z} \# \left\{ x\in J(0,N_3,0,N_4a,b,r+cq) \mid xb \equiv r + cq (a)\right\}\\
&=\sum_{\substack{a \le N_1\\ b \le N_2}}\sum_{\substack{c \in \Z\\ (a,b) \mid cq+r }} \frac{(a,b)}{a}\int_0^{N_3}1_{\{t \in \R \mid (r+cq)/b \le t \le (N_4 a + r + cq)/b\}}(t)dt + O(\mathcal{F} + \mathcal{E}_1 + \mathcal{E}_2)
\end{align*}
where the error term $\mathcal{E}_1$ accounts for the contribution from $E_H$ (we choose $H = [N^\lambda]-1$) for $a$, $b$ and $c$ such that
\begin{align}\label{ab_bound}
\min\left(\frac{a}{(a,r+cq)}, \frac{b}{(b,r+cq)} \right) > N^\lambda,
\end{align}
$\mathcal{F}$ is the exponential sum (we set $M = \max N_i$)
\begin{align*}
\sum_{c \in \widetilde{I}\cap \Z} \sum_{k \mid cq+r} \biggl\vert \sum_{N^\lambda \le \alpha \le \frac{N_1}{k}}\sum_{\substack{N^\lambda \le \beta \le \frac{N_2}{k}\\ (\alpha,\beta)=1}}\sum_{1 \le \vert h \vert \le H} C_{\max\left(0,\frac{r+cq}{\beta k}\right),\min\left(N_3,\frac{N_4\alpha k +r + cq}{\beta k}\right)}(h)e\left(\frac{-h(r+cq)\overline{\beta}}{\alpha k}\right)\biggr\vert,
\end{align*}
and $\mathcal{E}_2$ accounts for the the error coming from $a$, $b$ and $c$ not satisfying (\ref{ab_bound}) and the exponential sum terms with
\begin{align*}
\frac{\min(a,b)}{(a,b)} < N^\lambda.
\end{align*}
In particular these $a$, $b$ and $c$ satisfy
\begin{align*}
\min\left( \frac{a}{(a,r+cq)},\frac{b}{(b,r+cq)}\right) < N^\lambda.
\end{align*}
Thus $\mathcal{E}_2$ is estimated by
\begin{align*}
\mathcal{E}_2 &\ll N^\epsilon \sum_{c \in \widetilde{I}\cap \Z}\sum_{k
  \mid r+cq}\\
  &\phantom{=}\# \left\{(a,b)\in \N^2 \biggl| (a,b) = k, \ a \le N_1, \ b
  \le N_2, \ \min\left(\frac{a}{(a,r+cq)}, \frac{a}{(a,r+cq)}\right) < N^\lambda
  \right\}\\
&\ll N^\epsilon \sum_{c \in \widetilde{I}\cap \Z}\sum_{k
  \mid r+cq} \biggl(\frac{N_2}{k} \sum_{\delta_1 \mid \frac{r+cq}{k}} \#
  \{\alpha \in \N \mid \alpha \le N^\lambda \delta_1, \ (\alpha,
  (r+cq)/k)=\delta_1  \}\\
&\phantom{=}+\frac{N_1}{k} \sum_{\delta_2 \mid \frac{r+cq}{k}} \#
  \{\beta \in \N \mid \beta \le N^\lambda \delta_2, \ (\beta,
  (r+cq)/k)=\delta_2 \} \biggr)\\
&\ll \frac{N^{3+\lambda+2\epsilon}}{q}.
\end{align*}

We see that the error term $\mathcal{E}_1$ is at most of order 
\begin{align*}
\sum_{c \in \widetilde{I}\cap \Z} \sum_{k \mid cq+r} \sum_{\delta_1,\delta_2 \mid \frac{r+cq}{k}}
\sum_{(\alpha,\beta )\in \mathcal{A}} E_{H}\left(\max\left(0,\frac{r+cq}{\beta k}\right),\min\left(N_3,\frac{N_4\alpha k +r + cq}{\beta k}\right),\frac{\overline{\beta}(r+cq)}{k},\alpha\right),
\end{align*}
where
\begin{align*}
\mathcal{A} = \left\{(\alpha,\beta) \in \N^2 \biggl| \alpha \in I_1, \ \beta \in I_2, \ (\alpha,\beta) = 1, \ \left(\alpha,\frac{r+cq}{k }\right)= \delta_1, \ \left(\beta,\frac{r+cq}{k}\right) = \delta_2 \right\},
\end{align*}
$I_1 = [\delta_1 N^\lambda, N_1/k]$ and $I_2 = [\delta_2 N^\lambda , N_2/k]$. First we consider 
\begin{align*}
\sum_{(\alpha,\beta) \in \mathcal{A}} \theta_H\left(\frac{\overline{\beta}(r+cq)}{\alpha k}\right).
\end{align*}
From Lemma \ref{banks} it follows that
\begin{align*}
\sum_{\substack{\beta \in I_2\\ (\alpha,\beta) \in \mathcal{A}}} \theta_{H}\left(\frac{\overline{\beta}(r+cq)}{\alpha k}\right) \ll N^{-\lambda}\left(\frac{N_2\left(\alpha,\frac{r+cq}{k}\right)}{\alpha k}+1\right) \left(\frac{\alpha}{\left(\alpha,\frac{r+cq}{k}\right)}\right)^{1+\epsilon},
\end{align*}
and 
\begin{align*}
\sum_{\alpha \le \frac{N_1}{k}} \left(\frac{N_2\left(\alpha,\frac{r+cq}{k}\right)}{\alpha k}+1\right) \left(\frac{\alpha}{\left(\alpha,\frac{r+cq}{k}\right)}\right)^{1+\epsilon} & \ll N^{2+2\epsilon}.
\end{align*}
Thus it follows that
\begin{align*}
\sum_{c \in \widetilde{I}\cap \Z} \sum_{k \mid cq+r} \sum_{\delta_1,\delta_2 \mid \frac{r+cq}{k}} \sum_{(\alpha,\beta) \in \mathcal{A}} \theta_H\left(\frac{\overline{\beta}(r+cq)}{\alpha k}\right) \ll \frac{N^{4-\lambda+3\epsilon}}{q}.
\end{align*}
We now consider 
\begin{align*}
\sum_{(\alpha,\beta) \in \mathcal{A}} \theta_H\left(\frac{(\overline{\beta}-1/\beta)(r+cq)}{\alpha k}\right) = \sum_{(\alpha,\beta) \in \mathcal{A}}\theta_H\left(\frac{\overline{\alpha}(r+cq)}{\beta k}\right),
\end{align*}
where $\overline{\alpha}$ is the inverse of $\alpha$ modulo $\beta$. Again we apply Lemma \ref{banks} and obtain
\begin{align*}
\sum_{c \in \widetilde{I}\cap \Z} \sum_{k \mid cq+r}\sum_{(\alpha,\beta) \in \mathcal{A}} \theta_H\left(\frac{(\overline{\beta}-1/\beta)(r+cq)}{\alpha k}\right) \ll \frac{N^{4-\lambda+\epsilon}}{q}.
\end{align*}
The remaining terms of similar type are estimated in the same way. 

We now consider the exponential sum
\begin{align*}
\sum_{N^\lambda < \alpha \le \frac{N_1}{k}}\sum_{\substack{N^\lambda < \beta \le \frac{N_2}{k}\\ (\alpha,\beta)=1}}\sum_{1 \le \vert h \vert \le H} C_{\max\left(0,\frac{r+cq}{\beta k}\right),\min\left(N_3,\frac{N_4\alpha k +r + cq}{\beta k}\right)}(h)e\left(\frac{-h(r+cq)\overline{\beta}}{\alpha k}\right).
\end{align*}
First we look at 
\begin{align*}
\sum_{1 \le \vert h \vert \le H} \frac{1}{2\pi i h}\sum_{N^\lambda < \alpha \le \frac{N_1}{k}}\sum_{\substack{N^\lambda < \beta \le \frac{N_2}{k}\\ (\alpha,\beta)=1}} e\left(\frac{-h(r+cq)\overline{\beta}}{\alpha k}\right).
\end{align*}
Using standard exponential sum techniques we rewrite the inner sum (see e.g. \cite{banks05})
\begin{align*}
\sum_{\substack{N^\lambda < \beta \le \frac{N_2}{k}\\ (\alpha,\beta)=1}} e\left(\frac{-h(r+cq)\overline{\beta}}{\alpha k}\right) &= \sum_{\substack{\beta = 0\\ (\alpha,\beta)=1}}^{\alpha -1}e\left(\frac{-h(r+cq)\overline{\beta}}{\alpha k} \right) \sum_{N^\lambda < \zeta \le \frac{N_2}{k}} \frac{1}{\alpha} \sum_{\xi = -\lfloor \frac{\alpha-1}{2}\rfloor}^{\lfloor \frac{\alpha}{2}\rfloor} e\left(\frac{\xi(\beta-\zeta)}{\alpha} \right)\\
&=\frac{1}{\alpha}\sum_{\xi = -\lfloor \frac{\alpha-1}{2}\rfloor}^{\lfloor \frac{\alpha}{2}\rfloor}\sum_{\substack{\beta = 0\\ (\alpha,\beta)=1}}^{\alpha -1} e\left(\frac{\xi\beta-h\overline{\beta}(r+cq)/k}{\alpha}\right)\sum_{N^\lambda < \zeta \le \frac{N_2}{k}}e\left(-\frac{\xi \zeta}{\alpha} \right).
\end{align*}
The last sum is just a Weyl sum and is easily estimated (see e.g. \cite{iwaniec04} Section 8.2) by
\begin{align*}
\biggl\vert\sum_{N^\lambda < \zeta \le \frac{N_2}{k}}e\left(-\frac{\xi \zeta}{\alpha} \right)\biggr\vert \le \min\left( N_2,\left\vert \frac{\alpha}{\xi}\right\vert \right).
\end{align*}
Thus using the Weil bound for Kloosterman sums we obtain
\begin{align*}
\biggl\vert \sum_{\substack{N^\lambda < \beta \le \frac{N_2}{k}\\ (\alpha,\beta)=1}}& e\left(\frac{-h(r+cq)\overline{\beta}}{\alpha k}\right) \biggr\vert\\
 &\le \frac{1}{\alpha}\sum_{\xi = -\lfloor \frac{\alpha-1}{2}\rfloor}^{\lfloor \frac{\alpha}{2}\rfloor}\biggl\vert \sum_{\substack{\beta = 0\\ (\alpha,\beta)=1}}^{\alpha -1} e\left(\frac{\xi\beta-h\overline{\beta}(r+cq)/k}{\alpha}\right)\biggr\vert \biggl\vert\sum_{N^\lambda < \zeta \le \frac{N_2}{k}}e\left(-\frac{\xi \zeta}{\alpha} \right)\biggr\vert\\
&\ll \alpha^{-1/2+\epsilon}\left(\alpha, h\frac{r+cq}{k} \right)^{1/2}\biggl(N + \sum_{\substack{\vert \xi \vert \le \frac{\alpha}{2}\\ \xi \ne 0}} \left\vert \frac{\alpha}{\xi} \right\vert\biggr)\\
&\ll N^{1+2\epsilon}\alpha^{-1/2}\left(\alpha, h\frac{r+cq}{k} \right)^{1/2}
\end{align*}
and hence
\begin{align*}
\sum_{\alpha \le N_1} \alpha^{-1/2}\left(\alpha, h\frac{r+cq}{k} \right)^{1/2} &\le \sum_{d \mid h\frac{r+cq}{k}} \sqrt{d} \sum_{j \le \frac{N_1}{d}}(dj)^{-1/2} \ll \sum_{\substack{d \mid h\frac{r+cq}{k}\\ d \le N_1}}\sqrt{N_1/d} \ll N^{1/2+\epsilon}.
\end{align*}
The contribution to $\mathcal{F}$ is
\begin{align*}
\sum_{c \in \widetilde{I}\cap \Z} \sum_{k \mid cq+r} \sum_{1 \le h \le N^\lambda} \frac{N^{3/2+3\epsilon}}{h} \ll \frac{N^{7/2+5 \epsilon}}{q}.
\end{align*}
The remaining terms are handled in a similar way. Choosing $\lambda = 1/2$ yields the desired result.
\end{proof}

We need a slightly different version of the previous lemma as well
\begin{lemma}\label{error_lemma_ii}
Let $N \ge 1$ and $d_1,d_2 \mid k$. Then
\begin{align*}
\sum_{\substack{a_i \le M/d_i\\ (a_i,q/d_i) = 1}}&\sum_{c \in \Z} \# \left\{ x\in I(1,Nd_1/k,a_2,Na_2d_2,ka_1,cq) \mid xa_1 \equiv cq/k (a_2), \ (x, q/k) = 1\right\}\\
&=\frac{k\phi(q/k)}{q}\sum_{\substack{a_i \le M/d_i\\ (a_i,q/d_i) = 1}}\sum_{\substack{c \in \Z\\ (a_1,a_2) \mid c}} \frac{(a_1,a_2)}{a_2}\int_0^{N d_1/k}1_{\{t \in \R \mid cq/(ka_1) \le t \le (N a_2d_2 + cq)/(ka_1)\}}(t)dt+\\
&\phantom{==} O\left(\frac{N^{7/2+\epsilon}}{q}\right)
\end{align*}
uniformly for $M \asymp N$, $q\le N^{2-\delta}$ and $k \mid q$.
\end{lemma}
\begin{proof}
As in the proof of Lemma \ref{error_lemma_i} we see that
\begin{align*}
&\sum_{\substack{a_i \le M/d_i\\ (a_i,q/d_i) = 1}}\sum_{c \in \Z} \# \left\{ x\in I(1,Nd_1/k,a_2,Na_2d_2,ka_1,cq) \mid xa_1 \equiv cq/k (a_2), \ (x, q/k) = 1\right\}\\
&\phantom{===}=\sum_{\substack{a_i \le M/d_i\\ (a_i,q/d_i) = 1}}\sum_{c \in \Z} \# \left\{ x\in J(0,Nd_1/k,0,Na_2d_2,ka_1,cq) \mid xa_1 \equiv cq/k (a_2), \ (x, q/k) = 1\right\}+\\
&\phantom{=====}O\left(\frac{N^{3+\epsilon}}{q}\right).
\end{align*}
Assume $(q,l) = 1$ and $l \mid r$. Then
\begin{align*}
\# \{x \in \Z \mid a < x \le b, \ x \equiv r (q),\ (x,l) = 1 \} &= \# \{c \in \Z \mid (a-r)/q < c \le (b-r)/q,\ (c,l) = 1 \}\\
&= \sum_{d \mid l}\mu(d)\left(\left[\frac{b-r}{qd} \right]-\left[\frac{a-r}{qd} \right] \right)\\
&=\sum_{d \mid l}\mu(d) \# \{x \in \Z \mid a < x \le b, \ x \equiv r (qd)\}.
\end{align*}
Thus 
\begin{align*}
\# &\left\{ x\in J(0,Nd_1/k,0,Na_2d_2,ka_1,cq) \mid xa_1 \equiv cq/k (a_2), \ (x, q/k) = 1\right\} =\\
&\phantom{============}\sum_{d \mid \frac{q}{k}}\mu(d)\# \left\{ x\in J(0,Nd_1/k,0,Na_2d_2,ka_1,cq) \mid xa_1 \equiv cq/k (a_2)\right\}.
\end{align*}
We can now proceed as in the proof of Lemma \ref{error_lemma_ii}. The idea is the same so we only sketch the rest of the proof. 

For the $\theta_H$ sums we are lead to consider (essentially) sums of the form
\begin{align*}
\sum_{\vert c\vert \le \frac{N^2}{q}}\sum_{\delta \mid \frac{q}{k}} \sum_{\substack{\alpha_i \le M/d_i\\ (\alpha_i,q/d_i) = 1}} \theta_H\left(\frac{cq\overline{\alpha}_1}{k\delta \alpha_2} \right)
\end{align*}
and these can be estimated just as in Lemma \ref{error_lemma_ii} since (roughly speaking) the sum above just has fewer terms than the ones considered in the previous lemma (we exclude the terms that do not meet a certain coprimality condition) and all terms are non-negative. We have to be more careful with the exponential sum terms. We consider sums of the form
\begin{align}\label{exp_sum}
\sum_{ c \in I \cap \Z} \sum_{l \mid \frac{q}{k}}\sum_{\delta \mid c} \biggl\vert \sum_{\substack{H < \alpha_i \le \frac{M}{\delta d_i}\\ (\alpha_1,\alpha_2)=(\alpha_i,q/d_i)=1}}\sum_{1 \le \vert h \vert \le H} \frac{1}{h} e\left( -\frac{h cq \overline{\alpha}_1}{kl\delta \alpha_2}\right) \biggr\vert,
\end{align}
where $I$ is an interval of length at most of order $N^2/q$. In the inner sum we replace the coprimality condition $(\alpha_1,q/d_1) = 1$ with a M\"{o}bius sum (setting $\alpha_1 = \gamma \lambda$) and get
\begin{align*}
\sum_{\gamma \mid \frac{q}{d_1}}\mu(\gamma)\sum_{1 \le \vert h \vert \le H} \frac{1}{h} \sum_{\substack{H < \alpha_2 \le \frac{M}{\delta d_2}\\ (\alpha_2,\gamma q/d_2)=1}} \sum_{\substack{\frac{H}{\gamma} < \lambda \le \frac{M}{\gamma \delta d_1}\\ (\lambda,\alpha_2)=1}} e\left( -\frac{\overline{\gamma} h cq \overline{\lambda}}{kl\delta \alpha_2}\right).
\end{align*}
The inner sum is estimated as in Lemma \ref{error_lemma_i} by 
\begin{align*}
\sum_{\substack{H < \alpha_2 \le \frac{M}{\delta d_2}\\ (\alpha_2,\gamma q/d_2)=1}} \biggl\vert \sum_{\substack{\frac{H}{\gamma} < \lambda \le \frac{M}{\gamma \delta d_1}\\ (\lambda,\alpha_2)=1}} e\left( -\frac{\overline{\gamma} h cq \overline{\lambda}}{kl\delta \alpha_2}\right)\biggr\vert \ll N^{1/2+\epsilon} \sum_{\substack{\alpha_2 \le M\\ (\alpha_2,\gamma q/d_2)=1}} \left(\alpha_2,\frac{hcq}{lk\delta}\right)^{1/2} \ll N^{3/2+2\epsilon}.
\end{align*}
Thus we can estimate (\ref{exp_sum}) by $\frac{N^{7/2+3\epsilon}}{q}$.
\end{proof}

Using Lemma \ref{error_lemma_i} we can now prove the following:
\begin{theorem}\label{avg_i}
Let $K,M,N \ge 1$, $q \in \N$ with $M \asymp N$ and $N^{\eta} \le K$ for some $\eta > 0$. Assume also that $q \le N^{2 -\delta}$ for some $\delta > 0$ and $(\rho,q) = 1$. Then
\begin{align*}
\sum_{s=1}^q \left(\sum_{r \le K} \sum_{m\equiv \rho r + s(q)} \tau_{M,N}(m) - \frac{KMN}{q}\right)^2 \ll \frac{K^2N^{\max(7/2,4-\delta,4-\eta)+\epsilon}}{q}
\end{align*}
for any $\varepsilon > 0$. 
\end{theorem}
\begin{proof}
We may safely assume that $K \le q$. We see that
\begin{align*}
\sum_{s=1}^q \left(\sum_{r \le K} \sum_{m\equiv \rho r + s(q)} \tau_{M,N}(m)\right)^2 &= \sum_{\substack{r \le K\\ r' \le K}} \sum_{m\equiv m' + \rho (r-r')(q)} \tau_{M,N}(m)\tau_{M,N}(m')\\
&= \sum_{\vert l \vert \le K-1} R_l \sum_{m\equiv m' + l\rho (q)}\tau_{M,N}(m)\tau_{M,N}(m')
\end{align*}
where $R_l = [K]-\vert l \vert$. We consider the innermost sum
\begin{gather}
\begin{split}\label{tau_prod_i}
\sum_{m\equiv m' + l\rho (q)}&\tau_{M,N}(m)\tau_{M,N}(m')=\\
&\# \{(a_1,b_1,a_2,b_2) \in \N^4 \mid a_1,a_2\le N,\ b_1,b_2 \le M, \ a_1b_1 \equiv a_2b_2 + \rho l (q) \}
\end{split}
\end{gather}
We want to ``switch'' the roles of $q$ and $a_2$ and apply Lemma \ref{error_lemma_i}. Note that
\begin{align*}
a_1 b_1 \equiv a_2 b_2 + \rho l(q)
\end{align*}
exactly if there exists $c \in \Z$ such that
\begin{align*}
a_1 b_1  -\rho l -cq = a_2 b_2.
\end{align*}
Now fix $a_1$ and $a_2$. We see that (\ref{tau_prod_i}) is
\begin{align*}
\sum_{\substack{a_1,a_2 \le N\\ c \in \Z}}& \# \left\{ x \in I(1,M,a_2,Ma_2,a_1,\rho l + cq) \Bigm| xa_1 \equiv \rho l + cq (a_2)\right\}\\
&= \sum_{\substack{a_1,a_2 \le N\\ c \in \Z\\ (a_1,a_2)\mid cq + \rho l}}\frac{(a_1,a_2)}{a_2} \int_0^M 1_{\{t \in \R \mid (\rho l+cq)/a_1 \le t \le (\rho l+cq)/a_1+Ma_2/a_1 \}}(t)dt+O\left(\frac{N^{7/2+\epsilon}}{q} \right)\\
&= \sum_{\substack{a_2,a_1 \le N\\ (a_2,a_1,q)\mid l}}\frac{(a_2,a_1)}{a_2}\int_0^M \sum_{\substack{(ta_1 - l\rho)/q-Ma_2/q \le c \le (ta_1 - l\rho)/q\\ (a_2,a_1) \mid cq+\rho l}} 1 dt+O\left(\frac{N^{7/2+\epsilon}}{q} \right)\\
&=M\sum_{\substack{a_2,a_1 \le N\\ (a_2,a_1,q)\mid l}}\frac{(a_2,a_1)}{a_2} \frac{Ma_2(a_2,a_1,q)}{q(a_2,a_1)}+O\left(N^{2+\epsilon} + \frac{N^{7/2+\epsilon}}{q} \right)\\
&=\frac{M^2}{q}\sum_{\substack{a_2,a_1 \le N\\ (a_2,a_1,q)\mid l}}(a_2,a_1,q)+O\left(\frac{N^{\max(7/2,4-\delta)+\epsilon}}{q} \right).
\end{align*}
Again we consider the entire sum and see (using Lemma \ref{assymp_lemma})
\begin{align*}
\sum_{\vert l \vert \le K} R_l \sum_{\substack{a_2,a_1 \le N\\ (a_2,a_1,q)\mid l}}(a_2,a_1,q)
&= \sum_{a_1\le N}\sum_{\vert l \vert \le K} R_l \sum_{\substack{a_2 \le N\\ (a_2,a_1,q)\mid l}}(a_2,a_1,q)\\
&=2[K]\sum_{a_1=1}^N\left(KN+O((K+N)q^\varepsilon)\right) +O(KN^2q^\varepsilon)-\\
&\phantom{===}\sum_{a_1=1}^N\left(K^2N+O(K(K+N)q^\epsilon)\right)\\
&=K^2N^2+O(NK(N+K)q^\varepsilon).
\end{align*}
Thus
\begin{align*}
\sum_{s=1}^q \left(\sum_{r \le K} \sum_{m\equiv \rho r + s (q)} \tau_{M,N}(m) - \frac{KMN}{q}\right)^2 &= \sum_{\substack{r \le K\\ r' \le K}} \sum_{m\equiv m'+\rho (r-r') (q)} \tau_{M,N}(m)\tau_{M,N}(m')-\\
&\phantom{===}\frac{K^2M^2N^2}{q}+O\left(\frac{KN^4 + K^2 N^3}{q} \right)\\
&=O\left(\frac{K^2N^{\max(7/2,4-\delta,4-\eta)+\epsilon}}{q} \right).
\end{align*}
\end{proof}

This also shows that for the individual terms the asymptotics one should expect from Conjecture \ref{main_conj} holds for a subset of $\{s \in \N \mid 1 \le s \le q\}$ of full density in the following sense:
\begin{cor}\label{cor_fd}
Let $\nu > 0$ and assumptions be as in Theorem \ref{avg_i}. Then
\begin{align*}
\#\Big\{s\in \{ 1,\dots, q\} \Bigm| \Big\vert 1-  \frac{q}{KMN} \sum_{r \le K} \sum_{m\equiv \rho r + s(q)} \tau_{M,N}(m)\Big\vert > \nu \Big\} \ll qN^{-\min(\delta,1/2,\eta)+\varepsilon}
\end{align*}
uniformly in $K$, $q$ and $\rho$ for any $\varepsilon > 0$.
\end{cor}
In this connection it should be mentioned that if $s$ is ``bad'' then so is its neighbors. In fact we see that just one bad $s$ (in the sense of Corollary \ref{cor_fd}) will imply that there are at least $K^{1-\epsilon}$ bad values of $s$.

We now proceed to the proof of Theorem \ref{tau_individ}.
\begin{proof}[Proof of Theorem \ref{tau_individ}] Recall that $\phi(q)/q  \gg q^{-\epsilon}$ for any $\epsilon > 0$. Applying (\ref{r_coprime}) we see that
\begin{align*}
\sum_{(r,q) = k}\sum_{m\equiv  r (q)} \tau_{M,N}(m) &= \sum_{(m,q)=k} \tau_{M,N}(m)\\
&= \sum_{c \mid k}\phi(M/c, q/c) \phi(Nc/k,q/k)\\
&= \sum_{c \mid k}\frac{\phi(q/c)\phi(q/k)cMN}{q^2}+ O(q^\epsilon N).
\end{align*}
We proceed as in the proof of Theorem \ref{avg_i} and obtain using Lemma \ref{error_lemma_ii}
\begin{align*}
\sum_{(q,r) = k} \sum_{\substack{m \equiv r (q)\\ m' \equiv r (q)}}&\tau_{M,N}(m)\tau_{M,N}(m')\\ 
&= \sum_{d_1,d_2\mid k} \sum_{\substack{\alpha_1 \le \frac{M}{d_1}\\ (\alpha_1,q/d_1) = 1}}\sum_{\substack{\alpha_2 \le \frac{M}{d_2}\\ (\alpha_2,q/d_2) = 1}}\sum_{\substack{c \in \Z\\ (\alpha_1,\alpha_2)\mid c}}\\
&\phantom{==}\#\left\{x \in I\left(1,\frac{Nd_1}{k},\alpha_2,N\alpha_2 d_2, k\alpha_1,cq\right) \biggl| x \alpha_1 \equiv \frac{cq}{k} (\alpha_2), \ \left(x,\frac{q}{k}\right)=1 \right\}\\
&= \sum_{d_1,d_2\mid k}\sum_{\substack{\alpha_1 \le \frac{M}{d_1}\\ (\alpha_1,q/d_1) = 1}}\sum_{\substack{\alpha_2 \le \frac{M}{d_2}\\ (\alpha_2,q/d_2) = 1}}\sum_{\substack{c \in \Z\\ (\alpha_1,\alpha_2)\mid c}}
\frac{(\alpha_1,\alpha_2)\phi(q/k)k}{\alpha_2 q}\times\\
&\phantom{==} \int_0^{\frac{Nd_1}{k}} 1_{\{t \in \R \mid cq/(k\alpha_1) < t \le (cq+N\alpha_2 d_2)/(k \alpha_1)\}}(t)dt + O\left(\frac{N^{7/2+\epsilon}}{q}\right)\\
&= \frac{N^2\phi(q/k)}{q^2}\sum_{i=1,2}\sum_{d_i\mid k}d_1d_2\sum_{\substack{\alpha_1 \le \frac{M}{d_1}\\ (\alpha_1,q/d_1) = 1}}\sum_{\substack{\alpha_2 \le \frac{M}{d_2}\\ (\alpha_2,q/d_2) = 1}}1 + O\left(\frac{N^{\max(4-\delta,7/2)+\epsilon}}{q}\right)\\
&=\frac{M^2N^2\phi(q/k)}{q^2}\left(\sum_{d \mid k}d \phi(q/d) \right)^2+ O\left(\frac{N^{\max(4-\delta,7/2)+\epsilon}}{q}\right).
\end{align*}
This proves the theorem.
\end{proof}

Finally we prove Theorem \ref{main_conj_sup}.
\begin{proof}[Proof of Theorem \ref{main_conj_sup}]
We see that
\begin{align*}
\sum_{(\rho, q) = 1} \sum_{r \le K} \sum_{m\equiv \rho r(q)} \tau_{M,N}(m) &= \sum_{r \le K}\frac{\phi(q)}{\phi(q/(q,r))}\sum_{(m,q)=(r,q)} \tau_{M,N}(m)\\
&=\sum_{\substack{k\le K\\ k \mid q}} \frac{\phi(q)}{\phi(q/k)}\# \{r\in \N \mid r \le K, \ (r,q) = k\}\times\\
&\phantom{=}\#\{(a,b) \in \N^2 \mid a\le M,\ b \le N, \ (ab,q) = k\}.
\end{align*}
In the same way we see that
\begin{gather}
\begin{split}\label{ab_cop}
\#\{(a,b) \in \N^2 \mid a \le M,\ b &\le N, \ (ab,q) = k\}\\
&= \sum_{l \mid k} \#\{b\in \N \mid b \le N, \ (b,q) = l\}\times\\
&\phantom{=}\#\{a \in \N \mid a \le M, \ (a,q/l) = k/l\}\\
&=\sum_{l \mid k} \left(\frac{N}{q}\phi(q/l)+O(q^\varepsilon)\right)\left(\frac{Ml}{q}\phi(q/k)+O(q^\varepsilon)\right)\\
&=\frac{\phi(q/k)MN}{q^2}\sum_{l \mid k}l \phi(q/l)+O(Nq^{\varepsilon})\\
&\le\frac{MN \tau(k)}{k}+O(Nq^{\varepsilon}).
\end{split}
\end{gather}
Thus
\begin{gather}
\begin{split}\label{sum_ab}
\sum_{\substack{k\le K\\ k \mid q}} \#&\{(a,b) \in \N^2 \mid a\le M,\ b \le N, \ (ab,q) = k\}\\
 &= MN-\sum_{\substack{k > K\\ k \mid q}} \#\{(a,b) \in \N^2 \mid a\le M,\ b \le N, \ (ab,q) = k\}\\
&=MN+O\left(\frac{N^2q^\varepsilon}{K}+Nq^\epsilon \right).
\end{split}
\end{gather}
Combining (\ref{r_cop}) and (\ref{sum_ab}) it follows that
\begin{align}\label{mean}
\sum_{(\rho, q) = 1} \sum_{r \le K} \sum_{m\equiv \rho r(q)} \tau_{M,N}(m) =  \frac{\phi(q)KMN}{q} + O(N^2 q^\varepsilon + KNq^\epsilon).
\end{align}
Now we look at
\begin{align*}
\left(\sum_{r \le K} \sum_{m\equiv r\rho (q)}\tau_{M,N}(m) -\frac{KMN}{q}\right)^2
\end{align*}
and rewrite it as
\begin{align*}
&\left(\sum_{r \le K}\left( \sum_{m\equiv r\rho (q)}\tau_{M,N}(m) -\frac{MN}{q^2} \sum_{d \mid (r,q)}d\phi\left( \frac{q}{d}\right)\right)\right)^2+\left(\frac{MN}{q^2}\left(Kq-\sum_{r \le K}\sum_{d \mid (r,q)}d\phi\left( \frac{q}{d}\right) \right)\right)^2-\\
&2\left(\sum_{r \le K}\left( \sum_{m\equiv r\rho (q)}\tau_{M,N}(m) -\frac{MN}{q^2} \sum_{d \mid (r,q)}d\phi\left( \frac{q}{d}\right)\right)\right)\left(\frac{MN}{q^2}\left(Kq-\sum_{r \le K}\sum_{d \mid (r,q)}d\phi\left( \frac{q}{d}\right) \right)\right).
\end{align*}
Lemma \ref{Kq_lemma} implies that
\begin{align}\label{est_i}
\left\vert\frac{MN}{q^2}\left(Kq-\sum_{r \le K}\sum_{d \mid (r,q)}d\phi\left( \frac{q}{d}\right) \right)\right\vert \ll \frac{N^2(K+q) q^\epsilon}{q^2}.
\end{align}
We also see that Lemma \ref{Kq_lemma} and (\ref{mean}) implies
\begin{align}\label{est_ii}
\left\vert\sum_{(\rho,q) = 1}\sum_{r \le K}\left( \sum_{m\equiv r\rho (q)}\tau_{M,N}(m) -\frac{MN}{q^2} \sum_{d \mid (r,q)}d\phi\left( \frac{q}{d}\right)\right)\right\vert \ll K N^2 q^\varepsilon\left(\frac{1}{K}+\frac{1}{q} + \frac{1}{N}\right).
\end{align}
Thus it remains to look at
\begin{align}\label{new_mainterm}
\sum_{(\rho,q) = 1}\left(\sum_{r \le K}\left( \sum_{m\equiv r\rho (q)}\tau_{M,N}(m) -\frac{MN}{q^2} \sum_{d \mid (r,q)}d\phi\left( \frac{q}{d}\right)\right)\right)^2.
\end{align}
Using Cauchy-Schwarz inequality and Theorem \ref{tau_individ} we estimate (\ref{new_mainterm}) by
\begin{align*}
K \sum_{r \le K}&\sum_{(\rho,q) = 1}\left( \sum_{m\equiv r\rho (q)}\tau_{M,N}(m) -\frac{MN}{q^2} \sum_{d \mid (r,q)}d\phi\left( \frac{q}{d}\right)\right)^2\\
&=K\sum_{r \le K} \frac{\phi(q)}{\phi( q/(q,r))}\sum_{(s,q) = (r,q)}\left( \sum_{m\equiv s (q)}\tau_{M,N}(m) -\frac{MN}{q^2} \sum_{d \mid (r,q)}d\phi\left( \frac{q}{d}\right)\right)^2\\
&\ll \frac{KN^{\max(7/2,4-\delta)+\epsilon}}{q}\sum_{r \le K} \frac{\phi(q)}{\phi( q/(q,r))}\\
&\ll \frac{KN^{\max(7/2,4-\delta)+\epsilon}}{q}\sum_{r \le K} (q,r)\\
&\ll \frac{K^2N^{\max(7/2,4-\delta)+2\epsilon}}{q}
\end{align*}
Together with the estimates (\ref{est_i}) and (\ref{est_ii}) this proves the theorem.
\end{proof}

\section{Estimates for $\tau_M^*$ and $\tau_M$}
Through out this section we will restrict our discussion to $\tau_M^*$ though the results (with suitable modifications) clearly can be extended to cover $\tau_{M,N}$ as well.

We see that
\begin{align}\label{tau_as}
\sum_{m \le x} \tau_M(m) = \sum_{m \le x} \sum_{d=1}^{[M]} \delta_d(m) = \sum_{d=1}^{[M]} \frac{x}{d} + O(M) = x \log M + O(M+x).
\end{align}
Let $\tau$ denote the usual divisor function and note that
\begin{align*}
\tau_M^*(m) = \begin{cases} \tau(m) & {\rm if \ } m \le M\\ 2\tau_M(m)-\tau(m) & {\rm if \ } M < m < M^2\\ 0 & {\rm if \ } m \ge M^2 \end{cases}.
\end{align*}
It is well known that
\begin{align}\label{div_avg}
\sum_{m \le x} \tau(m) = x \log x + (2\gamma -1)x +O(\sqrt{x}).
\end{align}
Better estimates for the error term are known but that is irrelevant for our application. Thus
\begin{align*}
\sum_{m \le x}\tau_M^*(m) = \begin{cases} x \log x + (2\gamma -1)x +O(\sqrt{x}) & {\rm if \ } x \le M\\ x\log\frac{M^2}{x} + O(x) & {\rm if \ } M < x < M^2\\ [M]^2 & {\rm if \ } x \ge M^2 \end{cases},
\end{align*}
where the constants implied are absolute. 

Conjecture \ref{HB_conj} is probably hard to prove. We can however, give an estimate for the sum using a result due to Nair and Tenenbaum \cite{nair98}. Before we state the result we need to introduce some notation. Let $F:\N \to \R_+$. We say that $F \in \M(A,B,\epsilon)$ if $F$ satisfies (for $(m,n) = 1$)
\begin{align*}
F(mn) \le \min( A^{\Omega (m)}, Bm^\varepsilon)F(n)
\end{align*}
for some $A,B \ge 0$, where $\Omega (m)$ denotes the total number of prime factors of $m$, counted with multiplicity. The following theorem is an immediate consequence of \cite[Theorem~1]{nair98}:
\begin{theorem}\label{nair_tenenbaum}
Let $F \in \M(A,B,\frac{\epsilon\delta}{3})$, $0 < \delta \le 1$, $0 < \epsilon < \frac{1}{8}$. Then
\begin{align*}
\sum_{\substack{x < n \le x + y\\ n\equiv r (q)}} F(n) \ll \frac{y}{\phi(q)\log x} \sum_{\substack{n \le \frac{x}{q}\\(n,q)=1}} \frac{F(n)}{n}
\end{align*}
uniformly for $(r,q) = 1$, $x^{\frac{1+4 \epsilon\delta}{1+\delta}} \le y \le x$, $x \ge c_0 q^{1+\delta}$, where $c_0$ and the constant implied depends at most on $A$, $B$, $\delta$ and $\epsilon$.
\end{theorem}

Proposition \ref{tenenbaum_result} can be proved quite easily, since $\tau_N^*$ is closely related to the Hooley $\Delta$-function defined by
\begin{align*}
\Delta(n) = \max_{u \in \R} \#\{d \in \N \mid e^u < d \le e^{u+1}, \ d \mid n\}.
\end{align*}
We see that for all $N \ge 1$, $k \in \N$ we have
\begin{align}
\tau_N^*(m) \le 2k\Delta(m)
\end{align}
whenever $\frac{N^2}{2^{k}} < m \le \frac{N^2}{2^{k-1}}$. One easily checks that $\Delta \in \M(2,B,\epsilon)$ for any $\epsilon > 0$ and suitable $B$ (chosen according to $\epsilon$). In \cite{tenenbaum90} the following estimate was proved
\begin{align*}
\sum_{n \le x} \frac{\Delta (n)}{n}  \ll  e^{\sqrt{(2+\varepsilon)(\log\log x)(\log\log\log x)}}\log x.
\end{align*}
Using Theorem \ref{nair_tenenbaum} this implies that 
\begin{align}\label{Delta_est}
\sum_{\substack{m \equiv r
(q)\\ N^2/2^k < m \le N^2/2^{k-1} }} \Delta(m) \ll \frac{N^2}{2^k\phi(q)}e^{\sqrt{(2+\varepsilon)(\log\log N)(\log\log\log N)}}.
\end{align}
Now choose $l$ such that $c_0 q^{1+\delta} < N^2/2^l < N^{2-\delta'}$ for
some fixed $\delta' > 0$. Then
\begin{align*}
\sum_{m \equiv r (q)} \tau_N^*(m) \le \sum_{\substack{m \equiv r (q)\\
m \le N^2/2^l}} \tau(m) + 2\sum_{k=1}^l k \sum_{\substack{m \equiv r
(q)\\ N/2^k < m \le N^2/2^{k-1} }} \Delta(m).
\end{align*}
From (\ref{Delta_est}) and the Linnik-Vinogradov estimate (\ref{lv_thm}) it follows that
\begin{align*}
\sum_{m \equiv r (q)} \tau_N^*(m) \ll \frac{N^2}{q} + \frac{N^2 e^{\sqrt{(2+\varepsilon)(\log\log N)(\log\log\log N)}}}{\varphi(q)}\sum_{k=1}^\infty
\frac{k}{2^k}.
\end{align*}
This proves Proposition \ref{tenenbaum_result}.

Note that for $x \gg 1$
\begin{align}\label{log_e_est}
e^{\sqrt{(2+\varepsilon)(\log\log x)(\log\log\log x)}} \ll (\log x)^{\epsilon '}
\end{align}
for any $\epsilon ' >0$. To see this note that
\begin{align*}
\frac{\log\log\log x}{\log \log x} \to 0
\end{align*}
as $x \to \infty$. In particular 
\begin{align*}
\frac{\log\log\log x}{\log \log x} \le \frac{\epsilon'^2}{2+\epsilon}
\end{align*}
for any $\epsilon' > 0$ for $x$ large. Hence
\begin{align*}
(2+\epsilon) (\log\log x) (\log\log \log x) \le (\epsilon'\log\log x)^2.
\end{align*}
From this (\ref{log_e_est}) follows easily.

To prove Theorem \ref{tau_N_thm} we need the following lemma.
\begin{lemma}\label{tauN_lemma}
Let $\delta,\epsilon > 0$. Let $q\le x^{1-\delta}$ and $2 \le N$. Assume also that $N \ge q^{\epsilon}$. Then
\begin{align*}
\sum_{\substack{n \le x\\(n,q)=1}} \frac{\tau_N(n)}{n} \ll \frac{\phi(q)^2 \log x\log N}{q^2}.
\end{align*}
\end{lemma}
\begin{proof}
We see that
\begin{align}\label{lemma_eq}
\sum_{\substack{n \le x\\(n,q)=1}} \frac{\tau_N(n)}{n} = \sum_{\substack{ab \le x, \ a \le N\\ (a,q)=(b,q)=1}} \frac{1}{ab} = \sum_{\substack{a \le N\\ (a,q)=1}}\frac{1}{a} \sum_{\substack{b \le x/a\\ (b,q)=1}}\frac{1}{b}.
\end{align}
Thus we must consider
\begin{align*}
\sum_{\substack{a \le Y\\ (a,q)=1}}\frac{1}{a} &= \sum_{a \le Y}\frac{1}{a}\sum_{\substack{d\mid q\\ d \mid a}}\mu(d)\\
&= \sum_{d\mid q}\frac{\mu(d)}{d} \sum_{\alpha \le Y/d}\frac{1}{\alpha}\\
&= \sum_{d\mid q}\frac{\mu(d)}{d}\left(\log \frac{Y}{d} + \gamma +O\left(\frac{d}{Y}\right) \right)\\
&= (\log Y + \gamma)\frac{\phi(q)}{q} + O\left(\frac{\tau(q)}{Y}\right)-\sum_{d\mid q}\frac{\mu(d)\log d}{d}.
\end{align*}
Let $q'$ denote the square free part of $q$ and $p$ be a prime number. We see that
\begin{align*}
\sum_{d\mid q} \frac{\mu (d) \log d}{d} &= \sum_{d\mid q} \frac{\mu (d)}{d} \sum_{c \mid d} \Lambda (c)\\
&= \sum_{c \mid q'} \Lambda (c)\sum_{\substack{c\mid d\\ d\mid q'}} \frac{\mu (d)}{d}\\
&= \sum_{p \mid q'} \Lambda (p)\sum_{\delta\mid \frac{q'}{p}} \frac{\mu (\delta p)}{\delta p}\\
&= -\sum_{p \mid q'} \frac{\Lambda (p)}{p}\sum_{\delta\mid \frac{q'}{p}} \frac{\mu (\delta )}{\delta}.
\end{align*}
Thus
\begin{align*}
\sum_{d\mid q} \frac{\mu (d) \log d}{d} = O\left(\sum_{p \mid q'} \frac{\Lambda(p)}{p}\right) = O\left(\sum_{p \mid q} \frac{\log p}{p}\right).
\end{align*}
We split the last sum in two parts:
\begin{align*}
\sum_{p \mid q} \frac{\log p}{p} = \sum_{\substack{p \mid q\\ p \le (\log q)^2}} \frac{\log p}{p} + \sum_{\substack{p \mid q\\ p > (\log q)^2}} \frac{\log p}{p}.
\end{align*}
Clearly
\begin{align*}
\sum_{\substack{p \mid q\\ p > (\log q)^2}} \frac{\log p}{p} \le (\log q)^{-2}\sum_{p \mid q} \log p \le 1.
\end{align*}
We know that
\begin{align*}
\sum_{p \le x} \frac{1}{p} = O(\log \log x)
\end{align*}
Hence
\begin{align*}
\sum_{\substack{p \mid q\\ p \le (\log q)^2}} \frac{\log p}{p} \le 2\log \log q \sum_{p \le q} p^{-1} = O((\log \log q)^2).
\end{align*}
From this it follows that
\begin{align*}
\sum_{\substack{a \le Y\\ (a,q)=1}}\frac{1}{a} = (\log Y + \gamma)\frac{\phi(q)}{q} + O\left(\frac{\tau(q)}{Y}\right)+O((\log\log q)^2).
\end{align*}
Recall that $\phi(q) \gg \frac{q}{\log\log q}$. Thus it follows that
\begin{align*}
\sum_{\substack{a \le N\\ (a,q)=1}}\frac{1}{a} \ll \frac{\phi(q)\log N}{q}
\end{align*}
and
\begin{align*}
\sum_{\substack{b \le x\\ (b,q)=1}}\frac{1}{b} \ll \frac{\phi(q)\log x}{q}.
\end{align*}
The result now follows from (\ref{lemma_eq}).
\end{proof}

Theorem \ref{tau_N_thm} follows immediately from Theorem \ref{nair_tenenbaum} and Lemma \ref{tauN_lemma} since $\tau_N \in \M(2,B,\epsilon)$.

\section{Proof of Theorem \ref{lattice}}
We follow Section 5 in \cite{heathbrownPP}. For $\delta \in (0,1)$ define
\begin{align*}
\calR (M,\alpha,\delta) = \# \{(x,y) \in \Z^2 \mid \left\vert x\alpha-y\right\vert \le \delta, \ \vert x \vert \le M \}.
\end{align*}
The proof of Theorem \ref{lattice} is based on the following identity
\begin{align*}
\calS(M,N,K,p,q) = \sum_{a \le N} \calR (M,ap/q,K/q)  = N + 2\sum_{\vert r \vert \le K}\sum_{pm \equiv r(q)}\tau_{M,N}(m).
\end{align*}
We can transform it into a lattice point problem since
\begin{align*}
\{(x,y) \in \Z^2 \mid \left\vert x\alpha-y\right\vert \le \delta, \ \vert x \vert \le M \} = \{(x,y) \in \Z^2 \mid x\mathbf{u}+y\mathbf{v} \in [-\sqrt{M\delta},\sqrt{M\delta}]^2 \}
\end{align*}
where $\mathbf{u} = (\sqrt{\delta/M},\alpha \sqrt{M/\delta})$ and $\mathbf{v} = (0,-\sqrt{M/\delta})$. Since $\mathbf{u}$ and $\mathbf{v}$ generate a lattice of determinant $1$ we obtain
\begin{align*}
\calR (M,\alpha,\delta) = 4M\delta + O(\sqrt{M\delta}/\lambda_1) + O(1),
\end{align*}
where $\lambda_1$ is the length of the shortest non-zero vector in the lattice. We have $\delta = K/q$ and $\alpha = ap/q$ in our case. Thus we expect that the main term in $\calS(M,N,K,p,q)$ is $KMN/q$. 
It is the error term $O(\sqrt{M\delta}/\lambda_1)$ that needs attention. In particular one is concerned with the case where $\lambda_1$ is small. The idea is to consider $\sqrt{M\delta}/\lambda_1$ in dyadic intervals
\begin{align*}
E < \sqrt{M\delta}/\lambda_1 \le 2 E.
\end{align*}
Note that $E$ can be at most $M$ since $\lambda_1 \ge \sqrt{\delta/M}$. The $a$'s for which $E \le \sqrt{MK/q}$ contribute $N\sqrt{KM/q}$ to the error term for $\calS(M,N,K,p,q)$. Following Lemma 4 in \cite{heathbrownPP} the contribution from values of $a$ for which $E \ge \sqrt{MK/q}$ can be estimated by
\begin{align*}
\sum_{\sqrt{MK/q} \le E=2^k \le M} \sum_{1 \le F=2^h \le N} EFV(M/E,N/F,K/(qEF),p/q),
\end{align*}
where
\begin{align*}
V(A,B,D,\beta) = \# \left\{ (a,b,z) \in \N^2 \times \Z \mid a \le A, \ b \le B, \ (ab,z)=1, \ \vert ab\beta - z\vert \le D \right\}.
\end{align*}
Using Lemma 6 and Lemma 7 in \cite{heathbrownPP} we obtain the estimate
\begin{align*}
\sum_{\sqrt{MK/q} \le E=2^k \le M} \sum_{\substack{1 \le F=2^h \le N\\ EF \le (\log N)^{5/4}}} EFV(M/E,N/F,K/(qEF),p/q)  \ll N(KN/q)^{7/8}.
\end{align*}
Clearly we have
\begin{align*}
V(A,B,D,p/q) \le \sum_{\vert r \vert \le qD} \sum_{\substack{m \le AB\\ pm \equiv r (q)}}\tau_{A,B}(m) \le \sum_{\vert r \vert \le qD} \sum_{\substack{m \le AB\\ pm \equiv r (q)}}\tau(m).
\end{align*}
The previous results in the present paper suggest that there is a loss of roughly a factor $\log(AB)$ in the last inequality, but the last estimate will be sufficient for our purpose. 

We now need the fact that $p/q$ is of type $(2+\gamma,\calK)$. This implies
\begin{align*}
\left\vert \frac{p}{q}-\frac{z}{xy}\right\vert \ge \frac{1}{\calK (xy)^{2+\gamma}}
\end{align*}
unless $xy \mid q$ (remember that $(xy,z) = 1$). Thus if
\begin{align*}
\left\vert \frac{xyp}{q}-z\right\vert \le \frac{K}{qEF}
\end{align*}
we conclude that
\begin{align*}
(xy)^{1+\gamma} \ge \frac{qEF\calK}{K}
\end{align*}
unless $xy \mid q$. Thus for such $xy$ we can assume that
\begin{align*}
EF \le \left(\frac{(MN)^{1+\gamma}K}{q \calK}\right)^{\frac{1}{2+\gamma}}.
\end{align*}
Using the assumption (\ref{lattice_tech_cond}) we see that
\begin{align}\label{lv_cond}
\frac{MN}{EF} \ge \frac{MN}{\left(\frac{(MN)^{1+\gamma}K}{q\calK}\right)^{\frac{1}{2+\gamma}}} \gg q^{1+\frac{\delta(1+\gamma)}{2+\gamma}}.
\end{align}
It follows that
\begin{align*}
V(M/E,N/F,K/(qEF),p/q) \le \sum_{\vert r \vert \le \frac{K}{EF}} \sum_{\substack{m \le \frac{MN}{EF}\\ pm \equiv r (q)}}\tau(m) + \tau(q)^2.
\end{align*}
Using the Linnik-Vinogradov estimate (\ref{lv_thm}) one easily deduces that for $q \ll x^{1-\delta}$
\begin{align*}
\sum_{\vert r \vert \le S} \sum_{\substack{m \le x\\ m \equiv s (q)}}\tau(m) \ll \left(\tau(q)^2 + S (\log \log q)^2 \right) \frac{x \log x}{q},
\end{align*}
where the constant implied depends on $\delta$ only. Since (\ref{lv_cond}) holds we can use this result and we obtain
\begin{align*}
V(M/E,N/F,K/(qEF),p/q) \ll \left(\tau(q)^2 + \frac{K}{EF} (\log \log q)^2 \right) \frac{N^2 \log N}{EF q}.
\end{align*}
Clearly
\begin{align*}
\sum_{\sqrt{MK/q} \le E=2^k \le M} \sum_{\substack{1 \le F=2^h \le N\\ EF \ge (\log N)^{5/4}}} 1 \ll (\log N)^2.
\end{align*}
For $t \in (0,1)$ we recall the formula
\begin{align*}
\sum_{j=0}^n (j+1)t^j = \frac{(t-1)(n+1)t^n+1-t^{n+1}}{(1-t)^2}.
\end{align*}
Using this we see that
\begin{align*}
\sum_{1 \le E=2^k \le M} \sum_{\substack{1 \le F=2^h \le N\\ EF \ge (\log N)^{5/4}}} \frac{1}{EF} \ll \sum_{\frac{5}{4}\log \log N \le l \le 3\log N} \frac{l+1}{2^l} \ll \frac{\log\log N}{(\log N)^{5/4}}.
\end{align*}
This implies
\begin{align*}
\sum_{\sqrt{MK/q} \le E=2^k \le M} \sum_{\substack{1 \le F=2^h \le N\\ EF \ge (\log N)^{5/4}}} &EFV(M/E,N/F,K/(qEF),p/q)\\
&\ll \frac{N^2}{q}\left(\tau(q)^2(\log N)^3+\frac{K(\log \log N)^2}{(\log N)^{1/4}}\right).
\end{align*}


\begin{thebibliography}{99}
\bibitem{banks05} W. D. Banks, D. R. Heath-Brown and I. E. Shparlinski. \textit{On the average value of divisor sums
in arithmetic progressions}. Int. Math. Res. Not. (2005), no. 1, 1-25.
\bibitem{berry77} M. V. Berry and M. Tabor. \textit{Level clustering in the regular spectrum}. Proc. Roy. Soc. London A \textbf{356} (1977), 375-394.
\bibitem{blomer08} V. Blomer. \textit{The average value of divisor sums in arithmetic progressions}. Quart. J. Math. \textbf{59} (2008), 275-286.
\bibitem{friedlander85} J. B. Friedlander and H. Iwaniec. \textit{Incomplete Kloosterman sums and a divisor problem}. Ann. of Math. \textbf{121} (1985), no. 2, 319-350.
\bibitem{heathbrownPP} D. R. Heath-Brown. \textit{Pair correlation for fractional parts of $\alpha n^2$}. arXiv e-Print: math.NT/0904.0714v1.
\bibitem{heathbrown79} D. R. Heath-Brown. \textit{The fourth power moment of the Riemann zeta function}. Proc. London Math. Soc. \textbf{33} (1979), 385-422.
\bibitem{iwaniec04} H. Iwaniec and E. Kowalski. \textit{Analytic number theory}. AMS (2004).
\bibitem{linnik} Y. V. Linnik and A. I. Vinogradov. \textit{Estimate of the sum of the number of divisors in a short segment of an arithmetic progression}. Uspehi Mat. Nauk \textbf{12} (1957), no. 4 (76), 277-280.
\bibitem{marklof03} J. Marklof and A. Str\"{o}mbergsson. \textit{Equidistribution of Kronecker sequences along closed horocycles}. Geom. Funct. Anal. \textbf{13} (2003), no. 6, 1239-1280.
\bibitem{miller06} S. J. Miller and R. Takloo-Bighash. \textit{An invitation to modern number theory}. Princeton University Press (2006).
\bibitem{nair98} M. Nair and G. Tenenbaum. \textit{Short sums of certain arithmetic functions}. Acta Math. \textbf{180} (1998), 119-144.
\bibitem{rudnick98} Z. Rudnick and P. Sarnak. \textit{The pair correlation function of fractional parts of polynomials}. Comm. Math. Phys. \textbf{194} (1998), no. 1, 61-70.
\bibitem{rudnick01} Z. Rudnick, P. Sarnak and A. Zaharescu. \textit{The distribution of spacings between the fractional parts of $n^2\alpha$}. Invent. Math. \textbf{145} (2001), no. 1, 37-57.
\bibitem{shiu} P. Shiu. \textit{A Brun-Titschmarsh theorem for multiplicative functions}. J. Reine Angew. Math. \textbf{313} (1980), 161-170.
\bibitem{tenenbaum90} G. Tenenbaum. \textit{Sur une question d'Erd\"{o}s et Schinzel, II}. Invent. Math. \textbf{99} (1990), 215-224.
\bibitem{weyl16} H. Weyl. \textit{\"{U}ber die Gleichverteilung von Zahlen mod. Eins}. Math. Ann. \textbf{77} (1916), 313-352.
\bibitem{zaharescu03} A. Zaharescu. \textit{Correlation of fractional parts of $n^2\alpha$}. Forum Math. \textbf{15} (2003), 1-21.
\end{thebibliography}
\end{document}